\documentclass[oneside]{amsart}
\usepackage{enumerate,amssymb}
\usepackage{comment}
\usepackage{graphicx,xypic}
\newtheorem{thm}{Theorem}[section]
\newtheorem{coro}[thm]{Corollary}
\newtheorem{prop}[thm]{Proposition}

\newtheorem{claim}[thm]{Claim}
\newtheorem{lem}[thm]{Lemma}
\newtheorem*{lems}{Lemma}
\newtheorem{conj}[thm]{Conjecture}
\theoremstyle{definition}
\newtheorem{defn}[thm]{Definition}
\newtheorem{ex}[thm]{Example}

\newtheorem{question}[thm]{Question}

\newcommand{\Rset}{\mathbb{R}}

\newcommand{\Zset}{\mathbb{Z}}
\newcommand{\Nset}{\omega}
\newcommand{\Cset}{2^\Nset}
\newcommand{\CCset}{2^{<\Nset}}
\newcommand{\Pset}{\Nset^\Nset}
\newcommand{\UPset}{\Nset^{\uparrow\Nset}}

\newcommand{\ZZset}{\Zset^\Nset}

\newcommand{\eps}{\varepsilon}
\newcommand{\del}{\delta}
\newcommand{\subs}{\subseteq}
\newcommand{\sups}{\supseteq}
\newcommand{\clos}[1]{\overline{#1}}
\newcommand{\dist}{\underline{d}}
\newcommand{\concat}{\mkern-1mu^\smallfrown\mkern-4mu}
\newcommand{\rest}{{\restriction}}
\renewcommand{\leq}{\leqslant}
\renewcommand{\geq}{\geqslant}

\newcommand{\emany}{\exists^\infty}

\newcommand{\abs}[1]{\lvert#1\rvert}

\newcommand{\seq}[1]{\langle#1: n\in\omega\rangle}
\newcommand{\sseq}[1]{\langle#1\rangle}
\newcommand{\seqeps}{\seq{\eps_n}}
\newcommand{\Seqeps}{\seq{\eps_n}\in(0,\infty)^\omega}
\newcommand{\si}{$\sigma$\nobreakdash-}
\newcommand{\hm}{\mathcal H}

\DeclareMathOperator{\non}{\mathsf{non}}
\DeclareMathOperator{\add}{\mathsf{add}}
\DeclareMathOperator{\cov}{\mathsf{cov}}
\DeclareMathOperator{\cof}{\mathsf{cof}}

\DeclareMathOperator{\diam}{diam}
\DeclareMathOperator{\rng}{rng}
\DeclareMathOperator{\interior}{int}

\newcommand{\NN}{\mathcal N}
\newcommand{\MM}{\mathcal M}

\newcommand{\grG}{\mathbb{G}}

\newcommand{\eq}{\mathfrak{e\mkern-1.5mu q}}

\newcommand{\co}{\mathfrak{c}}
\newcommand{\dd}{\mathfrak{d}}
\newcommand{\nonN}{\non(\NN)}

\newcommand{\nonM}{\non(\MM)}
\newcommand{\covM}{\cov(\MM)}
\newcommand{\addM}{\add(\MM)}
\newcommand{\nonS}[1]{\non(\smz(#1))}

\newcommand{\UM}{\mathcal{U\mkern-2.5mu M}}

\newcommand{\mc}[1]{\mathcal{#1}}

\DeclareMathOperator{\hdim}{\dim_{\mathsf{H}}}
\newcommand{\uhm}{\overline{\mathcal H}{}}
\newcommand{\smz}{\ensuremath{\boldsymbol{\mathsf{Smz}}}}
\newcommand{\ssmz}{\ensuremath{\boldsymbol{\mathsf{Smz}}^\sharp}}

\DeclareMathOperator{\id}{id}
\DeclareMathOperator{\uhdim}{\overline{\dim}_{\mathsf{H}}}
\newcommand{\cyl}[1]{\langle#1\rangle}
\newcommand{\EE}{\mathcal E}
\newcommand{\NNs}{\mathcal N_\sigma}
\newcommand{\upto}{{\nearrow}}
\DeclareMathOperator{\proj}{proj}
\newcommand{\bbb}{\mathfrak b}
\newcommand{\nonSS}[1]{\non(\ssmz(#1))}

\newcommand{\eqq}{\mathfrak{e\mkern-1.5mu q}^*}
\newcommand{\addT}{\add^*(\MM)}
\newcommand{\TSI}{\ensuremath{\mathsf{TSI}}}
\newcommand{\CLI}{\ensuremath{\mathsf{CLI}}}
\newcommand{\equival}{\ensuremath{\Leftrightarrow}}

\newenvironment{enum}{\begin{enumerate}[\rm(i)]}{\end{enumerate}}
\newenvironment{itemyze}%
{\begin{list}{\textbullet}{\setlength{\labelwidth}{1ex}\setlength{\leftmargin}{2.1em}}}%
{\end{list}}

\newcommand{\Implies}{\ensuremath{\Rightarrow}}
\usepackage{color}

\begin{document}
%%%%%%%%%%%%%%%%%%%%%%%%%%%%%%%%%%%%%%%%%%%%%%%%%%%%%%%%%%%%%%%%%
%%%%%%%%%%%%%%%%%%%%%%%%%%%%%%%%%%%%%%%%%%%%%%%%%%%%%%%%%%%%%%%%%
\title
[Strong measure zero in Polish groups]
{Strong measure zero in Polish groups}

\author{Michael Hru\v s\'ak}
\address{Centro de Ciencias  Matem\'aticas, Universidad Nacional Aut\'onoma de M\'exico,
\' Campus Morelia, 58089, Morelia, Michoac\'an, M\'exico.}
\email{michael@matmor.unam.mx}
\urladdr{ http://www.matmor.unam.mx/~michael}

\author{Ond\v rej Zindulka}
\address
{Department of Mathematics\\
Faculty of Civil Engineering\\
Czech Technical University\\
Th\'akurova 7\\
160 00 Prague 6\\
Czech Republic}
\email{ondrej.zindulka@cvut.cz}
\urladdr{http://mat.fsv.cvut.cz/zindulka}

\thanks{%
The first author gratefully acknowledges support from \mbox{PAPIIT} grant
IN 100317.
The second author was supported from European Regional Development
Fund-Project ``Center for Advanced Applied Science''
(No.~CZ.02.1.01\slash 0.0\slash 0.0\slash 16\_019\slash 0000778).
}

%%%%%%%%%%%%%%%%%%%%%%%%%%%%%%%%%%%%%%%%%%%%%%%%%%%%%%%%%%%%%%%%%
%%%%%%%%%%%%%%%%%%%%%%%%%%%%%%%%%%%%%%%%%%%%%%%%%%%%%%%%%%%%%%%%%
\begin{abstract}
The notion of strong measure zero is studied in the context of
Polish groups. In particular, the extent to which  the theorem of
Galvin, Mycielski and Solovay holds in the context of an arbitrary
Polish group is studied.
Hausdorff measure and dimension is used to characterize strong measure zero.
The products of strong measure zero sets are examined.
Sharp measure zero, a notion stronger that strong measure zero, is shown
to be related to meager-additive sets in the Cantor set and Polish groups
by a theorem very similar to the theorem of
Galvin, Mycielski and Solovay.
\end{abstract}
%%%%%%%%%%%%%%%%%%%%%%%%%%%%%%%%%%%%%%%%%%%%%%%%%%%%%%%%%%%%%%%%%
%%%%%%%%%%%%%%%%%%%%%%%%%%%%%%%%%%%%%%%%%%%%%%%%%%%%%%%%%%%%%%%%%
\maketitle
%%%%%%%%%%%%%%%%%%%%%%%%%%%%%%%%%%%%%%%%%%%%%%%%%%%%%%%%%%%%%%%%%
%%%%%%%%%%%%%%%%%%%%%%%%%%%%%%%%%%%%%%%%%%%%%%%%%%%%%%%%%%%%%%%%%
\section{Introduction}
%%%%%%%%%%%%%%%%%%%%%%%%%%%%%%%%%%%%%%%%%%%%%%%%%%%%%%%%%%%%%%%%%
%%%%%%%%%%%%%%%%%%%%%%%%%%%%%%%%%%%%%%%%%%%%%%%%%%%%%%%%%%%%%%%%%

\emph{All spaces and topological groups considered are separable and metrizable.}

A natural extension of  a definition due to Borel (1919) \cite{MR1504785}
asserts that a metric space $X$ has
\emph{strong measure zero} ($\smz$) if for any sequence $\seq{\eps_n}$
of positive real numbers there is a cover $\{U_n:n\in\Nset\}$ of $X$
such that $\diam U_n \leq \eps_n$ for all $n$.

%\smallskip

In the same paper Borel conjectured that every strong measure zero set of reals is countable.
This was shown  to be independent
of the usual axioms of set theory by Sierpi\'nski (1928) \cite{sierpinski} and  Laver (1976)
\cite{MR0422027}. Later
it was observed by Carlson \cite{MR1139474} that the {\it Borel Conjecture} actually
implies a formally stronger statement that all separable $\smz$ metric spaces
are countable.

%\smallskip

We shall investigate the behaviour of strong measure zero sets  in arbitrary Polish groups.
In a sense we shall investigate the world, where the Borel conjecture fails, as most if not all of
our results are trivial if the Borel Conjecture holds.

%\smallskip

The subject of inquiry of this work starts with the theorem of
Galvin, Mycielski, and Solovay \cite{GMS,MR3696064} who, confirming  a conjecture
of Prikry,  proved that \emph{a set $A\subs\Rset$  is of strong measure zero
if and only if $A+M\neq\Rset$ for every meager set $M\subs\Rset$}.

%\smallskip

Relatively recently Kysiak \cite{kysiak} and Fremlin \cite{fremlin5},
independently, showed that an
analogous theorem is true for all locally compact metrizable groups
(see also \cite{wohofsky}). We present a proof of Kysiak and Fremlin's result
based on \cite{MR3453581}
and consider the natural question as to how far the result can be extended.
The fact that the theorem does not in general hold for all Polish groups was
established in \cite{MR3453581} and \cite{wohofsky} and extended in \cite{MR3707641}.
This depends on further set-theoretic axioms, as  the result obviously
holds for all Polish groups assuming, e.g., the Borel Conjecture.

%\smallskip

Cardinal invariants associated with strong measure zero sets on $\Rset$, $\Pset$, and
$\Cset$ have been studied rather extensively in recent decades
\cite{MR1350295, MR1253925, MR1955243,mejiaetal}. We review some of these, concentrating on the
\emph{uniformity} invariant of the \si ideal $\smz(\grG)$ of strong measure
zero subsets of a Polish group $\mathbb G$.   A version of the
Galvin-Mycielski-Solovay  theorem links  this
study to the investigation of the so-called \emph{transitive coefficient}
$\cov^*(\MM)$ in Polish groups
\cite{MR1350295, MR2224048, MR2421832}.

\bigskip

It was probably the aforementioned result of Prikry, Galvin, Mycielski and Solovay
that inspired a few notions of smallness on the real line and Cantor set
akin to strong measure zero. E.g., a set $S\subs\Rset$ is
\emph{strongly meager} if $S+N\neq\Rset$ for each Lebesgue null set $N$;
it is \emph{null-additive} if $S+N$ is Lebesgue null for each Lebesgue null set $N$; and
it is \emph{meager-additive} if $S+M$ is meager for each meager set $M$.
These notions easily extend to other Polish groups.

We will study the latter notion, which is obviously a strengthening of strong measure
zero.
Since the early nineties, meager-additive sets in the Cantor set receive quite some attention.
Let us single out the remarkable paper of Shelah~\cite{MR1324470} that provides a proof
that each null-additive set in the Cantor set $\Cset$ is meager-additive and also
the underlying combinatorial characterizations
of null-additive and meager-additive sets in $\Cset$ (cf.~\ref{ShelahM} below),
and Pawlikowski's paper~\cite{MR776210} providing fine combinatorics
and study of the so called transitive coefficients mentioned above that are actually
cardinal invariants of strong measure zero, meager-additive and null-additive sets,
and of course Bartoszy\'nski's book~\cite{MR1350295}.
However, all nontrivial results on meager-additive sets depended heavily on the
combinatorial and group structure of $\Cset$.
In 2009 Weiss~\cite{MR2545840,MR3241126}
found a method that made the theory transferrable to the real line.
Only very recently in~\cite{Zin_M-add,Zin_M-additive}
it was noted that there is a description of
meager-additive sets that resembles very much the Borel's definition
of strong measure zero. Metric spaces having this property
were termed to have \emph{sharp measure zero}.
This allowed for the theory of meager-additive sets to extend to other Polish groups.
We provide some highlights of the rather new theory of sharp measure  in metric spaces
and meager-additive sets, and sharp measure zero on $\Cset$
and on Polish groups, including calculation of the uniformity number
of sharp measure zero and meager-additive sets.

\bigskip

 Our set-theoretic notation is standard
and follows e.g. \cite{MR756630, MR1940513}. In particular, the set of finite
ordinals is identified with the set of non-negative integers and denoted interchangeably
by $\Nset$ and $\mathbb N$. In the same vein, the non-negative integers themselves are
identified with the set of smaller non-negative integers, in particular $2=\{0,1\}$.

All spaces considered are separable and
metrizable, often endowed with a compatible metric denoted $d$.
We denote by $B(x,\eps)$ the closed ball with radius $\eps$ centered at $x$,
the corresponding open ball will be denoted by $B^\circ(x,\eps)$.

The product spaces  of the type $A^\Nset$ for some
finite or countable set $A$ are  considered with the \emph{metric of least difference }
defined by
$d(f,g)=2^{-|f\wedge g|}$, where $f\wedge g= f\rest n$ for
$n=\min\{k: f(k)\neq g(k)\}$. The clopen balls in the space $A^\Nset$ are
represented by nodes
of the tree  $A^{<\Nset}$, given $s\in A^{<\Nset}$, we let
$\sseq s=\{f\in A^\Nset: s\subs f\}$. Given a subtree $T$ of $A^{<\Nset}$,
we let $[T]=\{ f\in A^\Nset: \forall n\in\Nset$ $f\rest n\in T \}$ be
the (closed) set of branches of $T$. A metric space is \emph{analytic}
if it is a continuous image of $\Pset$, and
 it is \emph{Borel} (\emph{absolutely $G_\delta$}, resp.) if it is Borel
 ($G_\delta$, resp.) in its completion.

A \emph{Polish group} is a separable, completely metrizable
topological group. A compatible metric $d$ on a separable metrizable group
$\grG$ is \emph{left-invariant} if $d(zx,zy)=d(x,y)$
for any $x,y,z\in\grG$.

A separable group $\grG$ is a \CLI{} group if it admits
a complete left-invariant compatible metric.
Abelian and locally compact  Polish groups are \CLI, while, e.g.,
the group $S_\infty$ of all permutations of $\Nset $ is not.

A separable group $\grG$ is a \TSI{} group if it admits
a (both-sided) invariant compatible metric.
Not every Polish group admits an invariant metric,
but if it is compact or abelian, then it does. Also, any invariant metric on a
Polish group is complete.

%%%%%%%%%%%%%%%%%%%%%%%%%%%%%%%%%%%%%%%%%%%%%%%%%%%%%%%%%%%%%%%%%
%%%%%%%%%%%%%%%%%%%%%%%%%%%%%%%%%%%%%%%%%%%%%%%%%%%%%%%%%%%%%%%%%
\section{Strong measure zero in Polish groups}
%%%%%%%%%%%%%%%%%%%%%%%%%%%%%%%%%%%%%%%%%%%%%%%%%%%%%%%%%%%%%%%%%
%%%%%%%%%%%%%%%%%%%%%%%%%%%%%%%%%%%%%%%%%%%%%%%%%%%%%%%%%%%%%%%%%
The notion of strong measure zero is in general neither a topological nor
a metric property, but a \emph{uniform} property; in particular,
a uniformly continuous image of a $\smz$ set is $\smz$, and  if $X$
uniformly embeds into $Y$, then any set $A\subs X$ that is not $\smz{}$
in $X$ is not $\smz{}$ in $Y$ either. \label{uniform}

 As all left-invariant (equiv right-invariant)
metrics on a separable metrizable group are uniformly equivalent the notion of
strong measure zero becomes seemingly ``topological'':
a subset $S$ of a topological group $\grG$ is \emph{Rothberger bounded} if
for every sequence $\seq{U_n}$ of neighbourhoods of $1_\grG$ there
is a sequence  $\seq{g_n}$ of elements of the group $\grG$ such that the
family $\seq{g_n \cdot U_n}$ covers $S$.  It follows  \cite{fremlin5}
that a subset of a Polish group $\grG$ is Rothberger bounded if and only if
it is %of
strong measure zero w.r.t.\ some (any) left-invariant metric on $\grG$.

Many of the results stated here could be phrased in the language of
uniformities and/or in terms of the property of being Rothberger bounded
(see \cite{fremlin5} for such treatment).

Whenever $\grG$ is a Polish group, $\smz(\grG)$ denotes the family of strong
measure zero sets with respect to any left-invariant metric
(i.e., the (left) Rothberger bounded sets as described above).

Of course, the choice of left-invariant over right-invariant is arbitrary, one
being isomorphic to the other via the inverse map of the group in question.
In fact, both the left Rothberger bounded and right Rothberger bounded set
form a \si ideal which is invariant under both left and right translations.

\begin{prop}\label{sigma}
$\smz(\grG)$ is a bi-invariant \si ideal.
\end{prop}

\begin{proof} To see that $\smz(\grG)$ is a \si ideal let
$\{X_n: n\in\Nset\}\subs \smz(\grG)$ and a sequence
$\{U_n: n\in\Nset\}$  of open subsets of $\grG$ be given.
Let $\{I_n: n\in\Nset\}$ be a partition of $\Nset$ into infinite sets.
As each $X_n$ is of strong measure zero, there is a sequence
$\{g_i: i\in I_n\}\subs \grG$ such that
$X_n\subs \bigcup_{i\in I_n} g_i\cdot U_i$.
Then $\bigcup_{n\in\Nset}X_n\subs \bigcup_{i\in\Nset}g_i\cdot U_i$.

Now, let $X\in \smz(\grG)$ and $g\in \grG$ be given.

To see that $g\cdot X\in \smz(\grG)$, note that if
$\{U_n: n\in\Nset\}$ is a sequence of open subsets of $\grG$  and
$\{g_i: i\in \Nset\}\subs \grG$ is such that
$X\subs\bigcup_{n\in\Nset} g_n\cdot U_n$, then
$g\cdot X\subs\bigcup_{n\in\Nset} g\cdot g_n\cdot U_n$.

To show that $X\cdot g\in \smz(\grG)$, let $\{U_n: n\in\Nset\}$ be a sequence
of open subsets of $\grG$. Consider
 the open sets $\{U_n\cdot g^{-1}: n\in\Nset\}$. As $X\in \smz(\grG)$ here is
 a sequence $\{g_n: n\in \Nset\}\subs \grG$ such that
$X\subs \bigcup_{n\in\Nset} g_n\cdot (U_n\cdot g^{-1})$.
Then $X\cdot g\subs  \bigcup_{n\in\Nset} g_n\cdot U_n$.
\end{proof}

Now, assuming Borel conjecture, or assuming that the group $\grG$ has an invariant
metric, the left and right
Rothberger bounded sets coincide. This is not true in general, though:

\begin{ex} Assuming {\sf CH}, there is a left Rothberger bounded subset
of the group of permutations $S^\infty$ of $\Nset$ which is not right Rothberger bounded.
\end{ex}

\begin{proof}
Denote by $\Omega$ the set of all finite partial injective functions from some
$n\in\Nset$ to $\Nset$.  Enumerate all sequences of elements of
$S_\infty$ as $\{y_\alpha:\alpha<\Nset_1\}$, and all increasing functions from
$\Nset$ to $\Nset $ as $\{f_\alpha:\alpha<\Nset_1\}$.

We shall recursively  construct
$\{g_\alpha:\alpha<\Nset_1\}\subs S_\infty$ and
$\{z_\alpha:\alpha<\Nset_1\}\subs S_\infty^\Nset$ so that
\begin{enumerate}
\item $\forall \beta<\alpha<\Nset_1 \ \exists n\in \Nset \ g_\alpha\rest f_\beta(n)=
  z_\beta(n)\rest  f_\beta(n)$, while
\item $\forall \beta<\alpha<\Nset_1 \ \forall n\in \Nset \ g_\alpha^{-1}\rest n+1
  \neq  y_\alpha^{-1}(n)\rest  n+1$, and
\item $\forall s\in \Omega$ and $m_0<m_1$ the first two elements of $\Nset\setminus \rng (s)$
$\exists a\in [\Nset]^{m_1+1}$
\begin{enumerate}
\item  $\forall n\in a$ $s\subs z_\alpha(n)\rest f_\alpha (n)$,
\item  $\forall n\in a$ $m_0 \in \rng(z_\alpha(n)\rest f_\alpha (n))$,
\item  $\forall n\in a$ $m_1 \not \in \rng(z_\alpha(n)\rest f_\alpha (n))$, and
\item $\forall i\neq j\in a$ $z_\alpha(i)^{-1}(m_0)\neq z_\alpha(j)^{-1}(m_0)$.
\end{enumerate}
\end{enumerate}
It should be clear, that if this can be accomplished then (1) guaranties that
the set $X=\{g_\alpha:\alpha<\Nset_1\}$ is of strong measure zero,
while (2) makes sure that $X^{-1}$ is not. The condition (3) is there for
the construction not to prematurely terminate.

Assume that $g_\beta, z_\beta$ for $\beta<\alpha$ have been constructed. First choose
$z_\alpha\in S_\infty^\Nset$ satisfying (3).
Then enumerate $\alpha=\{\beta_i:i\in \Nset\}$ and recursively find
$\{n_i:i\in\Nset\}$ so that 
$s_i=z_{\beta_i}(n_i)\rest f_{\beta_i}(n_i)$ satisfy
\begin{enum}
\item  $s_i\subs s_{i+1}$,
\item  if $m_i=\min (\Nset\setminus \rng (s_i))$ then $m_i\in \rng (s_{i+1})$,
\item $\forall n\leq m_i \ \exists k_n\leq n  \ k_n\in \rng (s_i)  \ s_i^{-1}(k_n)
  \neq y_\alpha^{-1}(n)(k_n)$.
\end{enum}

Then let $g_\alpha=\bigcup_{i\in\Nset}s_i$. Then $g_\alpha\in S_\infty$ satisfying (1)  by (i) and (ii), and (2) by (iii). 

To construct the sequence $\seq{s_n}$ start with $s_{-1}=\emptyset$. Having found $s_i$, let $m < k$ be the first two elements of $\omega\setminus \rng(s_i)$. By (3), there is $n_{i+1}\in\omega$ such that 
$s_{i+1}=z_{\beta_{i+1}}(n_{i+1})\rest f_{\beta_{i+1}}(n_{i+1})$ is such that $\{m, k\}\cap \rng (s_{i+1})=\{m\}$, and 
$s_{i+1}^{-1}(n)\neq y_{\alpha}(n)^{-1}(n)$ for every $m\leq n<k$.
\end{proof}

There is a close relation between strong measure zero and Geometric measure theory
which shall be explored later on in the text, in section~\ref{sec:haus}.
The first result in this direction
is due to Besicovitch  \cite{MR1555386,MR1555389} who showed that a set $X$ of reals has
strong measure zero if and only if every uniformly continuous image of $X$
has Hausdorff dimension $0$.

\smallskip

Here we shall characterize strong measure sets in Polish groups as exactly the
sets of {\it universal invariant submeasure zero}, a result due to J. Greb\'{\i}k.

It is a classical result of Haar \cite{MR1503103} that every locally compact
Polish group admits an (essentially  unique) left-invariant, countably additive,
 outer regular Borel  measure.  In a similar vein, we shall prove here that every
 Polish group admits a non-trivial countably subadditive, outer regular,
 left-invariant diffuse submeasure, a result used in the next section.

Recall that a function $\mu: \mathcal P(\grG)\to \Rset^+\cup\{\infty\}$ is
a submeasure if $\mu(\emptyset)=0$, and $\mu (A\cup B)\leq \mu(A)+\mu(B)$ whenever
$A,B$ are subsets of $\grG$. A submeasure $\mu$ on $\grG$ is
\begin{itemize}
\item {\it \si subadditive} if
  $\mu(\bigcup_{n\in\Nset}A_n)\leq \sum_{n\in\Nset}\mu(A_n)$, for any
  $\{ A_n:n\in\Nset\}\subs \mathcal P(\grG)$,
\item {\it outer regular} if $\mu(A)=\inf \{ \mu(U): \ A\subs U, U$ open in $X\}$,
  for any $A\subs\grG$,
\item {\it left-invariant} if $\mu(A)=\mu(g\cdot A)$, for any
  $A\subs \mathbb G$ and $g\in\grG$,
\item {\it non-atomic} or {\it diffuse} if $\mu(\{x\})=0$ for every $x\in \grG$, and
\item {\it non-trivial} if $\mu(\grG)>0$.
\end{itemize}

\begin{lem}\label{lemma:submeasure}
In every Polish group $\grG$ there is
a decreasing local basis $\{U_n:n\in\Nset\}$ of open sets  at $1_{\grG}$ such that
for every $m\in\Nset$ and $\{a_n: n>m\}\subs \mathcal P(\grG)$ such that
$|a_n|=n$ for every $n>m$, $U_m\not \subs \bigcup_{n>m} a_n\cdot U_n$.
\end{lem}

\begin{proof} Let $d$ be a left invariant compatible metric on $\grG$, and let
$e$ be a complete metric on $\grG$. Recursively choose the open sets $U_n$, $n\in\Nset$,
together with finite sets $b_n\subs U_n$ of size $n+1$ so that:
\begin{enumerate}
\item[In $d$:]  The points of $b_n$ are $3\diam U_{n+1}$ apart, and also
$3\diam U_{n+1}$ apart from the complement of $U_n$, while
\item[In $e$:] $\forall m<n$ $\forall \{g_i:m<i<n\}$ with $ g_i\in b_i$
$\diam \prod_{m<i<n} g_i\cdot U_n <\frac1n$.
\end{enumerate}

To verify that the sequence $\{U_n:n\in\Nset\}$ has the desired property assume
that $m\in \Nset$ and a sequence $\{a_n: n>m\}\subs \mathcal P(\grG)$ such
that $|a_n|=n$ for every $n>m$ are given.
Recursively choose $g_n\in b_n$ so that the set
$\prod_{m<i<n} g_i\cdot U_n\cap a_n\cdot U_n=\emptyset$. Such $g_n$ exists as
$d$ is left invariant, hence for every $g\in a_n$ the set $g\cdot U_n$ intersects
at most one of the sets $\prod_{m<i<n} g_i\cdot h \cdot U_n$ for $h\in b_n$,
and $|b_n|=|a_n|+1$. The closures of the sets $\prod_{m<i<n} g_i\cdot U_n$,
for $n>m$ form a decreasing sequence of sets of $e$-diameter converging to $0$,
hence by completeness of $e$ their intersection is a singleton $x\in U_m$ which
is not in $\bigcup_{n>m} a_n\cdot U_n$.
\end{proof}

\begin{thm}[\cite{MR3707641}] \label{submeasure}
There is a non-trivial, left-invariant, outer regular, \si sub\-additive
diffuse submeasure on every Polish group.
\end{thm}

\begin{proof} Fix a sequence $\{U_n:n\in\Nset\}$ as in Lemma \ref {lemma:submeasure}
and define for $A\subs \grG$:
$$
  \mu(A)=\inf\left\{\sum_{i\in\Nset} \frac1{n_i}: A\subs \bigcup_{i\in\Nset}
  g_i\cdot U_{n_i}\right\}.
$$
It is immediate from the definition that $\mu$ is a diffuse \si additive,
left invariant, outer regular submeasure on $\grG$. To see that $\mu$  non-trivial
it suffices to note that $\mu(U_m)=\frac1m$. To see that $\mu(U_m)$ is not
less than $\frac1m$, note that by the key property of $\{U_n:n\in\Nset\}$,
if $U_m\subs \bigcup_{i\in\Nset} g_i\cdot U_{n_i}$ then
$\sum_{i\in\Nset} \frac1{n_i}\geq\frac1m$.
\end{proof}

The promised characterization is the following:

\begin{thm}[J. Greb\'{\i}k, see \cite{MR3707641}] A subset $A$ of a Polish group $\grG$
is of left strong measure zero if and only if $\mu(A)=0$ for every  left-invariant,
outer regular, countably additive diffuse submeasure on $\grG$.
\end{thm}

\begin{proof}
Assume first that $X\in  \smz(\grG)$,  let $\mu $ be
a left-invariant, outer regular, countably additive diffuse submeasure on $\grG$,
and let $\eps>0$ be arbitrary. As $\mu$ is non-atomic and outer regular,
there is a sequence $\{U_n:n\in\Nset\}$ of neighborhoods of $1_\grG$
such that $\sum_{n\in\Nset}\mu(U_n)<\eps$.
Now, as $X\in  \smz(\grG)$, there is a sequence $\{g_n:n\in\Nset\}\subs\grG$
such that $X\subs \bigcup_{n\in\Nset} g_n\cdot U_n$. By left invariance of $\mu$,
$\mu(X)\leq\sum_{n\in\Nset} \mu(U_n)<\eps$. Hence $\mu(X)=0$.

On the other hand, assume that $X\subs \grG$ has $\mu(X)=0$ for every invariant,
non-atomic, outer regular submeasure $\mu$ on $\grG$, and let $\{V_n: n\in\Nset\}$
be a sequence of open neighbourhoods of $1_\grG$ in $\grG$. Let $\{U_n: n\in\Nset\}$
be a decreasing local basis as in Lemma \ref{lemma:submeasure}, that is such that
for every $m\in\Nset$ and $\{a_n: n>m\}$ such that $|a_n|=n$ for every $n>m$,
$U_m\not \subs \bigcup_{n>m} a_n\cdot U_n$, by passing on to a subsequence,
we may assume that $U_n\subs V_n$ for every $n\in\Nset$.
Let $\{n_i:i\in\Nset\}\subs \Nset$ be  such that $n_{i+1}>n_i$
for every $i\in\Nset$, and let
$\mathcal W=\{U_{n_{i+1}}: i\in\Nset\}$, and $w(U_{n_{i+1}})= \frac1{n_i}$. Then
define a submeasure $\mu$ by putting for $A\subs\grG$
$$
  \mu(A)=\inf\left\{\sum_{i\in\Nset} w(W_i): A\subs \bigcup_{i\in\Nset}
  g_i\cdot W_i\right\}
$$
with each $W_i\in\mathcal W$ and $g_i\in\grG$. This is again a left-invariant,
\si subadditive, non-atomic, outer regular submeasure, with
$\mu(U_{n_{i+1}})=\frac1{n_i}$. Hence $\mu(X)=0$, in particular, there is a sequence
$\{W_j:j\in\Nset \}$ and a sequence $\{q_j:j\in\Nset \}$ such that
$X\subs \bigcup_{j\in\Nset} g_i\cdot W_j$
and $\sum_{j\in\Nset} w(W_j)< 1$. This means that every $U_{n_{i+1}}$ appears
fewer that $n_i$-many times as one of the $W_j$, so there is permutation
$\pi\in S_\infty$ such that $W_{\pi(n)}\subs V_n$ for every $n\in\Nset$,
hence $X\subs \bigcup_{n\in\Nset} g_{\pi(n)}\cdot V_n$. Hence $X\in\smz(\grG)$.
\end{proof}

In the general context of a metric space, Szpilrajn \cite{MR1503125} proved
that every $\smz$ set $X$ has universal measure zero, i.e. has measure zero
for every finite diffuse Borel measure on $X$. It should be noted that unlike
strong measure zero sets, uncountable universal measure zero sets exist in
{\sf ZFC} as shown by Sierpi\'{n}ski and Szpilrajn\cite{SieSzpi}.

\begin{prop}[Szpilrajn \cite{MR1503125}]\label{prop:UMZ}
Strong measure zero sets in separable metric spaces have universal measure zero.
\end{prop}

\begin{proof}
Aiming towards contradiction, suppose that $X$ is $\smz$  yet there is a
diffused Borel measure $\mu$ on $X$ such that $\mu(X)=1$.
Define a function $f:(0,\infty)\to[0,1]$ by
$$
  f(r)=\sup\{\mu(E):\diam E\leq r\}.
$$
We claim that $\lim_{r\to0}f(r)=0$. Otherwise there is $\eps>0$ and a sequence
of sets $E_n$ such that $\diam E_n{\searrow} 0$ and $\mu(E_n)\geq\eps$.
Let $E=\bigcap_{n\in\Nset}\bigcup_{m\geq n}E_n$. Then clearly $\mu(E)\geq\eps>0$.
In particular $E\neq\emptyset$, i.e., there is $I\in[\Nset]^\Nset$ such that
$\bigcap_{n\in I}E_n\neq\emptyset$.
Suppose without loss of generality that $I=\Nset$. Since any two sets $E_n,E_m$
have a common point,
we have $\diam(\bigcup_{m\geq n}E_n)\leq 2\diam E_n$. Therefore
$\diam E\leq 2\lim_{n\to 0}\diam E_n=0$, which contradicts $\mu(E)>0$.
We proved that $\lim_{r\to0}f(r)=0$.
Therefore there is, for each $n\in\Nset$,  $\eps_n>0$ such that $\sum_nf(\eps_n)<1$.
Since $X$ is $\smz$, there are sets $U_n$ such that $\diam U_n<\eps_n$ that cover $X$.
It follows that
$$
  1=\mu(X)\leq\sum_n \mu(U_n)\leq\sum_n f(\diam U_n)\leq\sum_n f(\eps_n)<1,
$$
the desired contradiction.
\end{proof}

%%%%%%%%%%%%%%%%%%%%%%%%%%%%%%%%%%%%%%%%%%%%%%%%%%%%%%%%%%%%%%%%%
%%%%%%%%%%%%%%%%%%%%%%%%%%%%%%%%%%%%%%%%%%%%%%%%%%%%%%%%%%%%%%%%%
\section{The Galvin-Mycielski-Solovay Theorem in Polish groups}
%%%%%%%%%%%%%%%%%%%%%%%%%%%%%%%%%%%%%%%%%%%%%%%%%%%%%%%%%%%%%%%%%
%%%%%%%%%%%%%%%%%%%%%%%%%%%%%%%%%%%%%%%%%%%%%%%%%%%%%%%%%%%%%%%%%

In this section we study the Galvin-Mycielski-Solovay theorem
in the context of an arbitrary Polish group $\grG$. We denote by $\MM(\grG)$,
or simply by $\MM$ if there is no danger of confusion, the
ideal of meager subsets of  $\grG$.
Much of this section exists thanks to the following simple yet crucial
observation due to Prikry:

\begin{prop}[Prikry \cite{prikry}] \label{prop:triv}
Let $\grG$ be a separable group, and
let $S\subs\grG$ be such that $S\cdot M\neq \grG$ for all $M\in\MM(\grG)$. Then
$S\in \smz(\grG)$.
\end{prop}

\begin{proof}
Let $S$ be as above, and let  $\{U_n:n\in\Nset \}$ be a
family of open neighbourhoods of $1$ in $\grG$.
Let $\{g_n:n\in\Nset \}\subs \grG$ be such that
$U=\bigcup_{n\in\Nset} g_n \cdot U_n$ is dense open in $\grG$.
Then $U^{-1}$ is  dense open in $\grG$, the inverse being a homeomorphism, so
$M=\grG\setminus U^{-1}$ is nowhere dense in $\grG$. As $S\cdot M\neq \grG$,
there is $x\in\grG\setminus S\cdot M$, that is
$S\subs x\cdot U=\bigcup_{n\in\Nset} x\cdot g_n\cdot U_n$. Hence, $S\in\smz(\grG)$.
\end{proof}

 As mentioned in the introduction, Galvin, Mycielski,
and Solovay \cite{GMS,MR3696064} answered Prikry's question by showing that the reverse
inclusion holds
for $\Rset$. The same was recently proved  for all locally compact groups
by Kysiak \cite{kysiak} and Fremlin \cite{fremlin5}, independently.
We shall present a proof of their theorem (the converse of Prikry's result
for locally compact groups) here.
Our proof follows \cite{MR3453581}.

Call a subset $N$ of a topological group $\grG$
\emph{uniformly nowhere dense} if for every neighborhood $U$ of $1$ there is
a neighborhood $V$ of $1$  such that for every $x\in \grG$ there is a
$g\in \grG$ such that $g\cdot V\subs x\cdot U\setminus N$.
A set $M\subs \grG$ is \emph{uniformly meager} if it can be written
as a union of countably many
uniformly nowhere dense sets.
We denote the family of all uniformly meager
subsets of $\grG$ by $\UM(\grG)$ (or simply $\UM$).
The following generalizes~\cite[Theorem 4]{MR3696064}.

\begin{prop}[\cite{MR3453581}]\label{prop:nontriv}
Let $\grG$ be a Polish group which is either locally compact or
\TSI{}, and let  $S \in \smz(\grG)$.
Then $S\cdot M\neq \grG$ for all $M\in \UM(\grG)$.
\end{prop}
\begin{proof}
Assume first that $\grG$ admits a invariant metric~$d$. Recall that every
invariant metric on a Polish group is complete. Let $N$
be uniformly nowhere dense subset of $\grG$. Note that for every $y\in\grG$ and
an open set $U\subs\grG$,
$y\cdot U\cdot N=U\cdot y\cdot N$, and $y\cdot N$
is uniformly nowhere dense. It follows that for every uniformly nowhere dense
$N\subs \grG$
\begin{align}
  \forall U\text{open } \exists V\text{open }\forall x,y\in\grG\ \exists z\in\grG\
  \ z\cdot V\subs x\cdot U\setminus V\cdot y\cdot N.
\end{align}

Now, fix a $\smz$ set $S$ and a uniformly meager set $M$ written as the union of
an increasing sequence $\seq{N_n}$ of uniformly nowhere dense sets. Then
there is a sequence $\seq{U_n}$ of open subsets of $\grG$, such that for every $n>0$
\begin{align}
\forall x,y\in\grG\ \exists z\in\grG\ \ z\cdot U_n\subs x\cdot U_{n-1}
\setminus U_n\cdot y\cdot N_n.
\end{align}
As $S$ is  $\smz$, for every sequence $\{U_n: n\in\Nset\}$ of open sets
(of diameter converging to $0$)
there is a sequence $\{g_n: n\in\Nset\}\subs \grG$ such that each $s\in S$
is contained in infinitely many of
the sets $g_n\cdot U_n$. Applying (2) recursively there is
a
sequence $\seq{x_n}$ of elements of $\grG$ such that for every $n\in\Nset$
$$
  x_{n+1}\cdot g_{n+1}\cdot U_{n+1}\subs  x_n\cdot g_n\cdot U_n
  \setminus (g_{n+1}\cdot U_{n+1}\cdot N_{n+1}).
$$
The sequence $\seq{x_n}$ is Cauchy, let $x$ be its limit, i.e.,
$\{x\}=\bigcap_{n\in\Nset}  x_n\cdot g_n\cdot U_n$.
Then $x\not\in \bigcup_{n\in\Nset}  g_n\cdot U_n\cdot N_n\supseteq S\cdot M$,
 as the sequence $\seq{N_n}$ is increasing and
every element of $S$ is contained in infinitely many
of the $g_n\cdot U_n$ (for every $(s,m)\in S\times M$ there is an $n\in \Nset$
such that $s\in g_n\cdot U_n$ and $m\in N_n$).

\bigskip

Now if $\grG$ is locally compact, the proof proceeds along similar lines.
 Only (1) is replaced by the following lemma.

\begin{lem}\label{comp}
Let $\grG$ be a locally compact Polish group equipped with a complete metric $d$,
and let $U\subs \grG$ be an open set with compact closure $C=\clos U$ and
$P\subs \grG$ be compact nowhere dense. Then
$$
  \forall\eps>0\ \exists\del>0\ \forall x\in C\ \forall y\in K\ \exists z\in C
  \quad B(z,\del)\subs B(x,\eps)\setminus (B(y,\del)\cdot  P).
$$
\end{lem}
\begin{proof}
Fix $\eps>0$ and define
$f:C\times K\to\Rset$  by
$$
  f(x,y)=\sup\{t:\exists z\in C\ B(z,t)\subs B^\circ(x,\eps)\setminus y\cdot P)\}.
$$
Then $f$ is positive on $C\times K$ and attains its (positive) minimum.

To see that, consider, for each $z\in C$, the functions

\begin{align}
  g_z(x) &=\dist(z,X\setminus B^\circ(x,\eps)),\quad x\in C \notag \\
  h_z(y) &=\dist(z,y\cdot  P),\quad y\in K \notag
%
%\shortintertext{and note that}
\intertext{and note that}
  f(x,y) &=\sup_{z\in C}\min(g_z(x),h_z(y)).\label{fgh}
\end{align}

Using compactness it is easy to see that, for each $z\in C$, the function $h_z$ is
lower semicontinuous and that while $g_z$ does not have to be, it has the following
lower-semicontinuity property: if $x_n\to x$ and $g_z(x_n)\to 0$, then $g_z(x)=0$.

Now suppose that there are $(x_n,y_n)\in C\times K$ such that $f(x_n,y_n)\to0$.
Since $C,K$ are compact, passing to subsequences  we may assume
$(x_n,y_n)\to(x,y)\in C\times K$.
Use \eqref{fgh} and the semicontinuity properties of $g_z$ and $h_z$
to conclude that since $f(x_n,y_n)\to0$,
for any $z$ either $g_z(x_n)\to0$ and then $g_z(x)=0$, or else $h_z(y_n)\to0$
and then $h_z(y)=0$. Use \eqref{fgh} again to conclude that $f(x,y)=0$, the desired
contradiction proving that there is $\eta>0$ such that $f(x,y)>\eta$ for
all $x,y$.
It follows that
$$
  \forall x\in C\ \forall y\in K\ \exists z\in C
  \quad B(z,\eta)\subs B(x,\eps)\ \wedge B(z,\eta)\cap y\cdot P=\emptyset.
$$
The latter of course yields
$B(z,\frac\eta2)\cap B(y\cdot  P,\frac\eta2)=\emptyset$.
On the other hand,
there is $\xi>0$ such that
$$
  \forall y\in K\quad
  B(y,\xi)\cdot P\subs B( y\cdot P,\tfrac\eta2).
$$
It follows that $B(z,\frac\eta2)\cap B(y,\xi)\cdot P=\emptyset$.
Thus letting $\del=\min\{\frac\eta2,\xi\}$ yields the lemma.
\end{proof}

To conclude, write $\grG$ as the union of an increasing sequence of open sets with
compact closures
$K_n$,  and write a meager set $M$ as the union of an increasing sequence of
compact nowhere dense sets
$P_n$.
%Let
Choose $x_0\in \grG$ and $\eps_0>0$ such that $B(x_0,\eps_0)$ is compact.
Let $C=B(x_0,\eps_0)$.
By the above lemma there is a sequence $\Seqeps$ such that for every $n>0$
\begin{equation}\label{eq111}
  \forall x\in C\ \forall y\in K_n\ \exists z\in C
  \quad B(z,\eps_n)\subs B(x,\eps_{n-1})\setminus  B(y,\eps_n)\cap K_n)\cdot  P_n.
\end{equation}
We may of course suppose that $\eps_n\to0$.
Since $S$ is $\smz$, there is a cover $\{E_n\}$ of $S$ such that
$\diam E_n<\eps_n$ for all $n$ such that each point of $S$ is covered by infinitely many $E_n$'s).

For each $n$ there is $y$ such that
$E_n\subs B(y,\eps_n)$. Therefore, using repeatedly \eqref{eq111}, there is
a sequence $\seq{x_n}$ in $C$ such that for all $n\in\Nset$
$$
  B(x_{n+1},\eps_{n+1})\subs B(x_n,\eps_n)
  \setminus (E_{n+1}\cap K_{n+1})\cdot P_{n+1}.
$$
Let $x$ be the unique point of $\bigcap_{n\in\Nset}B(x_n,\eps_n)$
(there is one, since $B(x_0,\eps_0)$ is compact and is unique as $\eps_n\to0$).
Then $x\notin\bigcup_{n\in\Nset}(E_n\cap K_n)\cdot P_n$.

Thus, to prove that $x$ is not covered by $S\cdot M$ it suffices to show that
$S\times M\subs\bigcup_{n\in\Nset}(E_n\cap K_n)\times P_n$.
Let $(s,m)\in S\times M$. There is $k$ such that $(s,m)\in K_k\times P_k$.
Since there are infinitely many $n$ such that $s\in E_n$, there is $n\geq k$
such that $s\in E_n$,
hence $s\in E_n\cap K_k\subs E_n\cap K_n$. Also $m\in P_k\subs P_n$. Therefore
$(s,m)\in (E_n\cap K_n)\times P_n$. The desired inclusion is proved.
\end{proof}

The Galvin-Mycielski-Solovay/Fremlin/Kysiak result follows from the fact that
in a locally compact group every meager set is uniformly meager:

\begin{prop}[\cite{MR3453581}]\label{prop: M_equals_UM} A Polish group $\grG$ is
locally compact if and only if $\mathcal{M}(\grG)=\UM(\grG)$.
\end{prop}

\begin{proof}
We shall see first that {\it $\mathcal{M}(X)=\UM(X)$ for
every locally compact metric space~$X$}.

To that end it suffices to see  that every
nowhere dense subset of a compact space is, in fact, uniformly
nowhere dense. Let $N$ be a nowhere dense subset  of a compact space
$X$,  and  let $\eps>0$. Let $F$ be a finite subset of $X$ such that
$Z=\bigcup_{x\in F} B(x, \frac{\eps}{2})$.
For every $x\in F$ let $y_x\in B(x, \frac{\eps}{2})$ and $\delta_x>0$ be such that
$ B(y_x, \delta_x)\subs B(x, \frac{\eps}{2})\setminus N$. Then
$\delta=\min\{\delta_x:x\in F\}$ works as $B(x, \frac{\eps}{2})\subs B(z, \eps)$
whenever $z\in B(x, \frac{\eps}{2})$.

% Define a real-valued function $f$ on $K$ by putting
% $$
% f(x) =\sup\{\delta: \text{there is a } y\in X \text{ such that }
%  B(y,\delta)\subs B(x,\eps)\setminus N\}.
% $$
% The function $f$ is then easily seen to be continuous and $f(x)>0$ for every
% $x\in K$. By compactness of $K$, there is a
% $\delta>0$ such that $f(x)>\delta$ for every $x\in K$.
% That is the $\delta$ we were looking for.

\smallskip

On the other hand,  {\it $\mathcal{M}(X)\neq\UM(X)$ for every nowhere locally
compact complete metric space  $X$}.

To see this let $X$ be nowhere locally compact with a complete metric $d$.
Then for every $U$ with non-empty interior there is an $\eps^U>0$ and a
pairwise disjoint family $\{V_{k}^U:k\in\Nset\}$ of open balls of radius
$\eps^U$ contained in $U$.

We shall construct a nowhere dense set $N$ which is not uniformly meager.
To do that we recursively construct a family
$\{U_s: s\in \Nset^{<\Nset}\}$ of non-empty regular closed\footnote{Recall
that a set $U$ is \emph{regular closed} if
$U$ is the closure of the interior of $U$.} sets so that
\begin{enum}
 \item[(1)] $\diam U_s\leq 2^{-|s|}$ for every $s\in \Nset^{<\Nset}$,

 \item[(2)] $\bigcup_{n\in\Nset} U_{s^\smallfrown n}\subs U_s  $
 for every $s\in \Nset^{<\Nset}$,

 \item[(3)] $U_{s^\smallfrown n} \cap U_{s^\smallfrown m} = \emptyset$
 for every $s\in \Nset^{<\Nset}$ and any two distinct $m,n \in \Nset$,

 \item[(4)] $\interior (U_s\setminus \bigcup_{n\in\Nset} U_{s^\smallfrown n})\neq\emptyset$
 for every $s\in \Nset^{<\Nset}$,

 \item[(5)] for all $s\in \Nset^{<\Nset}$ and $k\in\Nset$ and $x\in V_{k}^{U_s}$
 there is a $n\in\Nset$ such that $ U_{s^\smallfrown n}
\subs B(x, 2^{-k})$.

 \end{enum}
To do this is straightforward.

Having constructed such a family,
let $N=\bigcap_{j\in\Nset}\bigcup_{|s|=j} U_s$. This is the required set:

It is nowhere dense as a non-empty open set $U$
is either disjoint from $N$, or contains $U_s$ for some  $s\in \Nset^{<\Nset}$.
Then, however,
$\emptyset \neq \interior (U_s\setminus \bigcup_{n\in\Nset}
U_{s^\smallfrown n})\subs U\setminus N$ by the property~(4) above.

Now we will prove that $N$ is not uniformly meager in $X$. The set $N$ is
naturally homeomorphic to $\Pset$ (see properties~(1)--(3) above),
hence satisfies the Baire Category Theorem.
Aiming toward a contradiction assume that $N\subs \bigcup_{l\in\Nset} N_l$,
where each $N_l$ is a closed uniformly nowhere dense subset of $X$. By the
Baire Category Theorem applied to $N$ there is an $s\in \Nset^{<\Nset}$
and an $l\in\Nset$ such that $U_s\cap N\subs N_l$, hence $U_s\cap N$ is uniformly
nowhere dense. So, there is a $\delta>0$ as in the definition of uniformly
nowhere dense corresponding  to $\eps^{U_s}$.
Consider $V_{k}^{U_s}$, for
$2^{-k}<\delta$.
Then, on the one hand
there is an $x\in V_{k}^{U_s}$ such that %$B(x, 2^{-n})\subs V _n^{U_s}\setminus N_s$,
$B(x, 2^{-k})\subs V_{k}^{U_s}\setminus N$,
and
on the other hand, there is
(see property~(5) above)
an $n\in\Nset$ such that %$\emptyset\neq N\cap U_{s^\smallfrown k}\subs N_n$,
$\emptyset \neq N \cap U_{s^\smallfrown n} \subs B(x, 2^{-k})$,
which is a contradiction.

 The result follows as every Polish group is either locally compact or nowhere locally compact.
\end{proof}

And finally:
\begin{thm}[Fremlin \cite{fremlin5}, Kysiak \cite{kysiak}]\label{kysiak}
Let $\grG$ be a locally compact Polish group. A
set $A\subs\grG$  is of strong measure zero
if and only if $A\cdot M\neq\grG$ for every meager set $M\subs\grG$
\end{thm}

\begin{proof}
The theorem follows directly from Propositions \ref{prop:triv},
\ref{prop:nontriv} and \ref{prop: M_equals_UM}.
\end{proof}

Next we shall discuss the possibility of extending the Galvin-Mycielski-Solovay theorem
to a larger class of Polish groups. First, one needs to realize that assuming the
Borel conjecture, the theorem holds trivially for every Polish group $\grG$,
as strong measure zero sets in all Polish groups are exactly the countable subsets
(\cite{MR1139474}), hence $S\cdot M$ is meager for every strong measure zero set $S$
 and  meager set $M$, hence $S\cdot M\neq\grG$.

On the other hand, it was shown in \cite{MR3453581}, that the Galvin-Mycielski-Solovay theorem
fails for the Baer-Specker group $\mathbb Z^\Nset$, assuming $\cov(\MM)=\co$.
We conjecture that assuming a strong failure of the Borel conjecture the locally
compact Polish groups are
exactly the ones for which the theorem holds.

\begin{conj}[CH]\label{CHSMZ}
 The Galvin-Mycielski-Solovay theorem holds in a Polish
group $\grG$ if and only if $\grG$ is locally compact.
\end{conj}

The Continuum Hypothesis is optimal in the sense that under {\sf CH} the
Galvin-Mycielski-Solovay Theorem fails for as many Polish groups as possible
follows from the logical complexity of the problem. The statement {\it $\grG$
satisfies the Galvin-Mycielski-Solovay Theorem} is a $\Pi^2_1$-statement with
$\grG$ as a parameter, and hence is decided by the $\Omega$-logic under the
Continuum hypothesis. Moreover, if the statement is true assuming {\sf CH}
it is true in {\sf ZFC} (see e.g. \cite{MR2069032}).

We shall verify (following \cite{MR3707641}) that the conjecture is true for Abelian
Polish groups, in fact, it is true for all groups with a complete
(both-sided)-invariant metric, and also for closed subgroups
of the permutation group $S_\infty$. Whether it is true in general remains open.

There is a, perhaps an even more interesting, stronger {\sf ZFC} conjecture
on the structure of Polish groups. The following
concept was introduced in \cite{MR3453581} and the term coined in \cite{MR3707641}:
A nonempty subset $C$ of a Polish group $\grG$  is said to be \emph{anti-GMS}
if it is nowhere dense and for every sequence $\{U_n : n \in \Nset\}$ of open
neighborhoods of $1$ there is a sequence $\{g_n : n \in\Nset\}$ of elements
of $\grG$ such that for every $g\in \grG$, the set
$g\cdot \bigcup_{n\in\Nset}g_n\cdot U_n$  is dense in $C$.

The reason for introducing anti-GMS sets is the following:

\begin{prop}[\cite{MR3453581}]\label{anti-GMS}
Assuming $\cov(\MM)=\co$,
if $C\subs\grG$ is anti-GMS, then there is a strong measure zero set $S$
such that $S\cdot C=\grG$.
\end{prop}

\begin{proof} Enumerate $\grG=\{g_\alpha:\alpha<\co\}$ and enumerate
all sequences of open sets in $\grG$ as $\{\seq{U_n^\alpha}:\alpha<\co\}$.
Let $C\subs\grG$ be anti-GMS, and for every $\alpha<\co$ let
$\seq{g_n^\alpha}$ be such that
for all $g\in\grG$, $(g \cdot \bigcup_{n\in\Nset}g_n^\alpha\cdot U_n^\alpha) \cap M$
is comeager in $M$. Let $U_\alpha=\bigcup_{n\in\Nset}g_n^\alpha\cdot U_n^\alpha$.

As $\covM=\co$, the intersection of fewer than $\co$  relatively
dense open subsets of $M$ is not empty. In particular,
for every $\alpha<\co$, there is an
$$
  m_\alpha\in M\cap (g_\alpha^{-1} \cdot \bigcap_{\beta\leq\alpha}U_\beta).
$$
There is then an $x_\alpha\in \bigcap_{\beta\leq\alpha}U_\beta$ such that
$m_\alpha= g_\alpha^{-1} \cdot x_\alpha$.
That is $g_\alpha= x_\alpha \cdot m_\alpha^{-1}$.

Let $X=\{x_\alpha:\alpha<\co\}$. Then %$\grG= M^{-1}\cdot X$, so
$\grG = X \cdot M^{-1}$. Let us see that $X\in \smz(\grG)$:
Given a sequence $\{U_n: n\in\Nset\}$ of neighbourhoods of $1_\grG$,
consider the subsequence
$\{U_{2n}: n\in\Nset\}$. It is listed as $\seq{U_n^\alpha}$ for some
$\alpha<\co$.
By the construction, every $x_\gamma\in U_\alpha$ for $\gamma\geq\alpha$.
On the other hand, $X\setminus U_\alpha\subs\{x_\beta:\beta <\alpha\}$,
hence has size less than $\covM=\co$. Committing the sin of forward reference,
by Theorem~\ref{thm:Roth}(i),  $X\setminus U_\alpha$ is a $\smz$ set,
hence can be covered by the sets $\{U_{2n+1}:n\in\Nset\}$. Hence,
 $X$ has strong measure zero.
\end{proof}

The anti-GMS sets are our only tool for disproving the Galvin-Mycielski-Solovay
Theorem in non-locally compact groups. Hence the {\it Strong Conjecture} is:

\begin{conj} Exactly one of the following holds for any Polish group $\grG$:
Either $\grG$ is locally compact, or it contains an anti-GMS set.
\end{conj}

It is immediate from Proposition \ref{anti-GMS} that the Strong conjecture,
indeed, solves the GMS conjecture stated above. It is, in fact the Strong conjecture
we have verified for the aforementioned classes of groups:

\begin{thm}[\cite{MR3707641}] Let $\grG$ be a non-locally compact \TSI{} Polish group.
Then $\grG$ contains an anti-GMS set.
\end{thm}

\begin{proof} Let $\mu$ be a non-trivial left-invariant countably subadditive,
diffuse outer regular submeasure
given by Theorem \ref{submeasure}.

\begin{claim} There is a nowhere dense set $C\subs \grG$ such that for every
open set $O$ intersecting $C$ there is an open set $U$ and
$\{g_m:m\in\Nset\}\subs \grG$ such that
for every $m\in \Nset$
$g_m\cdot U\subs O$ and $\lim_{m\in\Nset}\mu (g_m\cdot U\setminus C)=0$.
\end{claim}

\begin{proof} To construct the set $C$ let $\{B_n:n\in\Nset\}$ be a basis for
the topology of $\grG$. Recursively construct an increasing sequence of open sets
$\{W_k: k\in\Nset\}$ and a sequence $\{A_k: k\in\Nset\}$ of countable sets of
pairs of the form $\langle U,\eps\rangle$, where $U$ is an open subset of
$\grG$ and $\eps>0$ such that
\begin{enumerate}
\item $B_k\cap W_{k+1}\neq \emptyset$,
\item $\forall g\in\grG$ $|\{  \langle U,\eps\rangle\in A_k: g\in U   \}|\leq k$,
\item $\forall   \langle U,\eps\rangle\in A_k$
  $\mu(\overline{W_{k+1}} \cap U\setminus \overline W_k)<\frac{\eps}{2^k}$,
  and if  $k$ is the least such that $ \langle U,\eps\rangle\in A_k$
  then $\overline W_k\cap U=\emptyset$, and
\item
\begin{enumerate}
\item either $B_k\subs W_{k+1}$,
\item or, there is  an open neighborhood $V$ of $1_\grG$ and  there are distinct
$\{g_i:i\in\Nset\}\subs \grG$
and  $\{\eps_i:i\in\Nset\}$ such that for all $i\in\Nset $
$ g_i\cdot V\subs B_n$, $\langle  g_i\cdot V,\eps_i\rangle\in A_k$ and
$\lim _{i\in\Nset}\eps_i=0$.
\end{enumerate}
\end{enumerate}

To carry out the construction, put first $W_0=A_0=\emptyset$.
Having constructed $W_k$ and $A_k$, see first whether $B_k\subs \overline W_k$.
If so, let $A_{k+1}=A_k$ and let $W_{k+1}=W_k\cup B_k$.
If $B_k\not \subs \overline W_k$, let $U_0, U_1$ be disjoint open subsets of
$B_k\setminus \overline W_k$. By (2) there is an open set $U_2\subs U_0$
contained or disjoint from every $U$ appearing in $A_k$, by non-atomicity of
$\mu$ we may require $\mu(U_2)$ to be so small that $W_{k+1}=W_k\cup U_2$
satisfies (3) for all  $ \langle U,\eps\rangle\in A_k$ such that
$U_2\subs U$. Finally, as $U_1$ is not compact, there are infinitely many balls
of the same diameter (i.e. translates of the same open set) with disjoint closures
contained in $U_1$, add them to $A_{k+1}$ paired with some real numbers converging
to $0$. It is clear that (1)-(4) are satisfied.

Now, let $C=\grG\setminus \bigcup_{k\in\Nset} W_k$.
 Then $C$ is a closed nowhere dense set by (1). By (3),
$\mu(U\setminus C)<\eps$ for every
$\langle U,\eps\rangle\in \bigcup_{k\in\Nset} A_k$, and if
$B_k\cap C\neq\emptyset$ then by (4) $B_k$ contains infinitely many sets
with the required properties.
\end{proof}

{\it The set $C$ from the claim is anti-GMS}. In order to verify this let a sequence
$\{U_m:m\in\Nset\}$ of open neighborhoods of $1_\grG$ be given,
without loss of generality of diameter shrinking to $0$.
Let again $\{B_n:n\in\Nset\}$ be a basis for the topology of $\grG$, and let
$\{A_n:n\in\Nset\}$ be a partition of $\Nset$ into infinite sets. For every
$B_n$ such that $B_n\cap C\neq\emptyset$ let $W_n$ be an open neighborhood of
$1_\grG$
such that there are distinct $\{g_i:i\in\Nset\}\subs \grG$ such that
$g_i\cdot W_n\subs B_n$ and
$\lim _{i\in\Nset}\mu(g_i\cdot W_n\setminus C)=0$.

Now, let $V_n$ be an open neighborhood of $1_\grG$
such that  $V_n\cdot U_j\subs W_n$ for all but finitely many $j\in A_n$,
and let $g_j, h_j\in \grG$ be such that
$h_j\cdot W_n\subs B_n$, $\mu (h_j\cdot W_n\setminus C)<\mu(U_j)$, and
$G=\bigcup_{j\in A_n} h_j\cdot V_n\cdot g_j^{-1}$ (here is where we use the
invariance of the metric). The sequence $\{g_i:i\in\Nset\}\subs \grG$
witnesses (for the sequence $\{U_i:i\in\Nset\}$ that $C$ is anti-GMS.
Indeed, of $g\in\grG$
and $B_n\cap  C\neq \emptyset$ then there is a $j\in A_n$ such that
$g\in h_j\cdot V_n\cdot g_j^{-1}$, i.e.
$g\cdot g_j\in h_j\cdot V_n$, hence
$$
  g\cdot g_j\cdot U_j\subs h_j\cdot V_n\cdot U_j\subs h_j\cdot W_n\subs B_n.
$$
As $\mu(h_j\cdot W_n\setminus C)<\mu(U_j)$, we get
$C\cap B_n\cap g\cdot g_j\cdot U_j\neq \emptyset$, as required.
\end{proof}

\begin{coro} The strong conjecture is true for Abelian groups.
\end{coro}

We do not know, whether the strong conjecture is true for all Polish groups,
but we can confirm it for another important class of groups -- the automorphism
groups of countable structures:

\begin{thm}[\cite{MR3707641}]
Let $\grG$ be a non-locally compact closed subgroup of $S_\infty$.
Then $\grG$ contains an anti-GMS set.
\end{thm}

\begin{proof} As $\grG$ is not locally compact, there is an infinite $A\subs \Nset$
such that for all $n\in A$ there are infinitely many $m\in \Nset$  for which
there are $g\in G$ such that $g\rest n=1_\grG\rest n$ and $g(n)=m$.
Given $n,m\in A$, $n<m$ define
$$
  R_{n,m}=\{(g,h)\in\grG\times\grG: g(n)\not\in \rng(h\rest m)
  \text{ and }h(n)\not\in \rng(g\rest m)\}.
$$
Note that the relation is left-invariant. Let
$B=\{u\in \Nset^{<\Nset}: U$ is one-to-one and $dom(u)\in  A$
and $u\subs g$ for some $g\in \grG\}.$

\begin{claim}
For every  $n<m\in A$ and $u\in B$ such that $|u|=n$ there are
$g,h\in\grG$ both extending $u$ such that $(g,h)\in R_{n,m}$.
\end{claim}

\begin{proof} By the left-invariance of $R_{n,m}$, we may assume that $u$
is the identity on $n=dom(u)$. Let
$\{g_k: k\in\Nset\}\subs \grG$ be such that $g_k\rest n=u$ and
$\{g_k(n)): k\in\Nset\}$ are pairwise distinct, by further shrinking we may assume
that for every $l\in [n,m)$ either $\{g_k(n)): k\in\Nset\}$ are pairwise distinct
or all equal. In particular, for every $k_0\in\Nset$ the set
$\{k_1\in\Nset : g_{k_0}(n)\in \rng (g_{k_1}\rest [n,m)\}$ is finite,
hence $(g_0, g_k)\in R_{n,m}$ for almost all $k\in \Nset$.
\end{proof}

Given a subset $H$ of $B$ let
$T(H)=\{u\in B:\forall k\in \Nset\ u\rest k\not \in H\}.$
Construct $D\subs B$  such that
\begin{enumerate}
\item $\forall u\in B\ \exists v\in D\ u\subs v$,
\item $\forall v,v'\in D\ \rng (v)\subs \rng(v')$ or $\rng (v')\subs \rng(v)$, and
\item $\forall u\in T(D) \ \forall m\in \Nset \ \exists v\in T(D)\ m\in \rng(v)$.
\end{enumerate}
It is easy to construct such a set using a simple bookkeeping argument.
Having done so, let $C=[T(D)].$

To see that $C$ is anti-GMS, consider an infinite set $Z\subs A$.
By the Claim there is a sequence
$\{g_n:n\in Z\}\subs \grG$ such that
$$
  \forall u\in B\ \exists n_o<n_1 \in Z\ u\subs g_{n_0}\cap g_{n_1}
  \text{ and } (g_{n_o},g_{n_1})\in R_{|u|,n_0}.
$$
To finish the proof it suffices to show that
$g\cdot \bigcup_{n\in Z} \langle g_n\rest n\rangle \cap C$
is dense in $C$ for every $g\in \grG$.
To that end fix $g\in \grG$ and $v\in B$ such that
$\langle v\rangle \cap C\neq\emptyset$, and let $k=dom(v)$ and $u=g^{-1}\cdot v$.
Choose $n_0<n_1\in Z$ such that
$u\subs g_{n_0}\cap g_{n_1} \text{ and } (g_{n_o},g_{n_1})\in R_{k,n_1}$.
Then either $\langle g\cdot g_{n_0}\rest n_0\rangle\cap C\neq \emptyset$ or
$\langle g\cdot G_{n_1}\rest n_1\rangle\cap C\neq \emptyset$:
If not, then there are $s_0,s_1\in D$  such that
$s_0\subs g\cdot g_{n_0}\rest n_0$ and
$s_1\subs g\cdot g_{n_1}\rest n_1$.
Neither $s_0\subs v$ nor $s_1\subs v$ as
$C\cap \langle v\rangle\neq\emptyset$, so
$g\cdot g_{n_0}(k)\in \rng(s_1)$ and $g\cdot g_{n_1}(k)\in \rng(s_0)$.
This, however, contradicts the assumption that $\rng(s_0)\subs \rng(s_1)$ or
$\rng(s_1)\subs \rng(s_0)$.
\end{proof}

\section{Cardinal invariants of \smz{} in Polish groups}
\label{sec:cardinal}

Given an ideal   $\mc I$ of subsets of a set $X$ the following are the standard
cardinal invariants associated with $\mc I$:
\begin{align*}
  \non(\mc I)&=\min\{\abs{Y}:Y\subs X\wedge Y\notin\mc I\},
  \\
  \add(\mc I)&=\min\{\abs{\mc A}:\mc A\subs\mc I
  \wedge{\textstyle\bigcup}\mc A\notin\mc I\},
  \\
  \cov(\mc I)&=\min\{\abs{\mc A}:\mc A\subs\mc I
  \wedge{\textstyle\bigcup}\mc A =X\},
  \\
  \cof(\mc I)&=\min\{\abs{\mc A}:\mc A\subs\mc I
  \wedge(\forall I\in\mc I)(\exists A\in\mc A)(I\subs A)\}.
\end{align*}
We denote by $\MM,\NN$ the ideals of meager and Lebesgue null subsets of $\Cset$,
respectively. For $f,g\in\Pset$, we say that
$f\leq^* g$ if $f(n)\leq g(n)$ for all but finitely many $n\in\Nset$
(the order of \emph{eventual dominance}).
A family $F\subs\Pset$ is \emph{bounded} if there is an $h\in\Pset$ such that
$f\leq^* h$ for all $f\in F$; and $F$ is \emph{dominating} if for any
$g\in\Pset$ there is $f\in F$ such that $g\leq^*f$.
The cardinal invariants related to eventual dominance are $\mathfrak b$ %,
(the minimal cardinality of an unbounded family) and $\mathfrak d$
(the minimal cardinality of a dominating family).

We shall briefly review the results (not necessarily in the chronological order)
concerning other cardinal invariants of $\smz(\grG)$, after which we give a more
detailed account of the $\non(\smz)$ in Polish groups. We shall denote by $\smz$
the ideal of strong measure zero subsets of $\Rset$.
Concerning additivity of $\smz$, Carlson  \cite{MR1139474} in effect showed that
$\add(\NN)\leq \add(\smz)$, that $\add(\smz)\leq\non(\NN)$ is a triviality,
while Goldstern, Judah and Shelah \cite{MR1253925} showed that consistently
$\cof(\MM)<\add(\smz)$, and, of course, Laver \cite{MR0422027} that consistently
$\add(\smz)< \mathfrak b$ and Baumgartner \cite{MR823775} that
consistently $\add(\smz)< \non(\NN)$.

For cofinality of $\smz$, there are lower bounds $\cov(\NN)$ and $\cov(\MM)$ (see below),
and it is folklore fact that assuming {\sf CH} $\cof(\smz)>\co$,
while the Borel conjecture produces models where $\cof(\smz)=\co$.
Yorioka \cite{MR1955243} and more recently Cardona \cite{cardona} produced models
of {\sf ZFC} where $\cof(\smz)<\co$. According to our knowledge,
it has not been subject to study if $\add(\smz(\grG))$, and/or $\cof(\smz(\grG))$
may vary depending on the Polish group in question.

It is, however, known that the uniformity numbers may differ depending on the group.
There is also a surprising asymmetry between $\cov(\smz)$ and $\non(\smz)$.
The trivial lower bound for $\cov(\smz)$ is $\cov (\NN)$ (every $\smz$-set has
Lebesgue measure zero) and it also seems to be the best one.  Pawlikowski \cite{MR1056381}
showed that $\cov(\smz)< \add(\MM)$ is consistent. On the other hand, Cardona, Mej\'{\i}a and
Riera-Marid \cite{mejiaetal}   recently showed that
$\cov(\smz)=\Nset_2=\co$ in the iterated Sacks model, hence,
there seems to be very little in terms of upper bounds on $\cov(\smz)$,
in particular, consistently $\cof(\NN)<\cov(\smz)$.

Before moving on we shall mention some of the open problems concerning these invariants:

\begin{question}
\begin{enumerate}
\item (\cite{mejiaetal}) Is it consistent that all four of the cardinal invariants
corresponding to $\smz$ have simultaneously different values?
\item  (\cite{mejiaetal}) Is it consistent that $\add(\smz)<\min\{\cov(\smz),\non(\smz)\}$?
\item  Do any of $\add(\smz(\grG))$, $\cov(\smz(\grG))$, $\cof\smz(\grG))$
depend on which Polish group one considers?
\end{enumerate}
\end{question}

Finally, we arrive at the uniformity of $\smz(\grG)$ which we shall discuss in
considerably more detail. Two more invariants are required here
(see \cite{MR2224048,MR613787,MR1350295,MR3453581}):

\begin{align*}
  \eq&=\min\{\abs{F}:F\subs\Pset\text{ bounded, }
  \forall g\in\Pset\ \exists f\in F\ \forall n\in\Nset\ f(n)\neq g(n)\}.\\
  \eqq&=\min\{\abs{F}:F\subs\Pset\text{ bounded, }\\
  &\qquad\qquad\forall g,h\in\Pset\ \exists f\in F\ \emany n\ \forall k\in[h(n),h(n+1))\ f(n)\neq g(n)\}.
\end{align*}
It is a theorem of Bartoszy\'nski  \cite[2.4.1]{MR1350295} that omitting
``bounded'' from the definition of $\eq$ yields $\covM$, i.e.
 $$\cov(\MM)=\min\{\abs{F}:F\subs\Pset \
  \forall g\in\Pset\ \exists f\in F\ \forall n\in\Nset\ f(n)\neq g(n)\}.$$
The  following diagram
(see \cite{MR1350295, MR2648159} for proofs) summarizes the provable inequalities between
the cardinal invariants mentioned (the arrows point from the smaller to the larger cardinal).

$$
\xymatrix{
\bbb \ar[r]  & \dd  \\
\add(\MM) \ar[r] \ar[u] \ar[rd] & \covM \ar[r] \ar[u] & \eq \ar[r] & \non(\NN)\\
& \eqq \ar[ru]
}
$$
In addition, $\add(\MM) =\min\{\mathfrak b, \eq\} = \min\{\mathfrak b, \eqq\}$,
%and $\cov(\MM) =\min\{\mathfrak d, \ed\}$,
while $\cov(\MM)< \min\{\mathfrak d, \eq\}$ is
consistent with ZFC by a theorem of Goldstern, Judah and Shelah \cite{MR1253925}.
For any separable metric space $X$ there are upper and lower bounds for $\nonS X$
given by Rothberger~\cite{MR0004281} and
Szpilrajn~\cite{Szpilrajn34}, respectively. The uniformity invariant
$\nonS X$ for $X=\Cset$ and $X=\Pset$
was calculated by Bartoszy\'nski \cite{MR1350295},
and Fremlin and Miller~\cite{MR954892}, respectively.

\begin{thm}\label{thm:Roth}
Let $X$ be a separable metric space that is not \smz{}.
\begin{enum}
\item{}
{\rm(\cite{MR0004281})} $\covM\leq\nonS X$,
\item{}
{\rm(\cite{Szpilrajn34})} if $X$ is not of universal measure zero%
\footnote{Recall that a metric space $X$ is of
\emph{universal measure zero} if there is no probability Borel measure on $X$
vanishing on singletons.}, then $\nonS X\leq\nonN$,
\item{}
{\rm(\cite{MR1350295})} $\nonS \Pset=\covM$ and
\item{}
 {\rm(\cite{MR954892})} $\nonS \Cset=\eq$.
\end{enum}
\end{thm}

 \begin{proof}
 (i) by now is standard: Given a separable metric space $X$ and a sequence
 $\{\eps_n: n\in\Nset\}$, pick a
 countable dense set $\{d_n: n\in\Nset\}\subs X$. For every one to one function
 $f\in \Pset$ let
 $U_f=\bigcup_{n\in\Nset} B(x_n, \eps_{f_n})$.
 Now, assume that $|X|<\cov(\MM)$. To finish the proof it suffices to note that
 for every $x\in X$  the set $N_x=\{f\in \Pset: f$ is one-to-one and
 $x\not \in U_f\}$ is nowhere dense in the (closed) subspace of $\Pset$
 consisting of one-to-one functions.

 \smallskip

 To see (ii) recall that every $\smz$ set is of universal measure zero by
 Proposition \ref{prop:UMZ}, so
it suffices to show that $\non(\NN)$ is the minimal size of a space which is
not of universal measure zero. To see this note that any diffuse Borel probability
measure $\mu$ on a separable metric space $X$ extends to a finite Borel diffuse
probability measure $\overline\mu $ on its completion $\hat X$ by putting $\overline \mu(A)=\mu (A\cap X)$.
Now, by a theorem of Parthasarathy \cite{MR0651013} there is a measure
preserving Borel isomorphism between $\hat X$ with $\overline \mu$ and $[0,1]$
equipped with the Lebesgue measure, hence $|X|\geq \non(\NN)$.

 \smallskip

  For (iii) it suffices to see that $\nonS \Pset\leq \covM$. Let $F\subs\Pset$
  be such that $|F|=\covM$ and
  $\forall g\in\Pset\ \exists f\in F\ \forall n\in\Nset\ f(n)\neq g(n)$.
  We shall show that $F\not\in\smz(\Pset)$. Let $\seq{s_n}$ be a sequence
  of elements of $2^{<\Nset}$ such that $|s_n|=n+1$. Define $g\in\Pset$
  by putting $g(n)=s_n(n)$. Then there is an
   $f\in F$ such that $ f(n)\neq g(n)$ for all $n\in\Nset$. This, however, means,
   that $s_n\not \subs f$  for any $n \in\Nset$. That is no sequence of open sets
   of diameter $\frac{1}{2^{n+1}}$ covers $F$.

 \smallskip

 The proof of (iv) is similar. First we shall show that $\nonS {\Cset}\leq\mathfrak{eq}$.
 To that end let $X\subs \Cset$ be of size less that $\mathfrak{eq}$,
 and let a sequence $\seq{\eps_n}$ of positive real numbers be given.
 Let $h\in\Pset$ be such that $\frac{1}{2^{h(n)}}\leq \eps_n$
 for every $n\in\Nset$. For each $n\in\Nset$ enumerate $2^n$  -- the set of
 binary sequences of length $n$ -  as $\{s^n_m:m< \Cset\}$.
 For every $x\in X$ let $f_x\in\Pset$ be defined by
$$f_x(n)=m\text{ if and only if } x\rest h(n)=s^{h(n)}_m.$$
Then $f_x(n)\leq 2^{h(n)}$ for every $x\in X$ and $n\in\Nset$.
As $|X|<\mathfrak{eq}$, there is a $g\in\Pset$ such that
$f_x\cap g\neq\emptyset$ for every $x\in X$, and without loss of generality, $g(n)\leq 2^{h(n)}$ for every
$n\in\Nset$ (values above are irrelevant, and can be changed).
Then $\seq{\langle s^{h(n)}_{g(n)}\rangle}$ covers $X$.

On the other hand, assume that $F\subs\Pset$ is a bounded family of
size less than $\nonS {\Cset}$. Let $h\in\Pset$ be such that
$f(n)\leq 2^{h(n)}$ for every $f\in F$ and $n\in\Nset$.
Let $\{s^n_m:m < \Cset\}$ enumerate $2^n$ as above.
For every $f\in \Pset$ let
$$
  x_f=s_{f(0)}^{h(0)}{}^{\smallfrown}
  s_{f(1)}^{h(1)}{}^{\smallfrown}s_{f(2)}^{h(2)}{}^{\smallfrown} \dots
$$
and consider the set $X=\{x_f:f\in F\}$. As $X$ is $\smz$, there is a sequence $\seq{t_n}$ such that
\begin{enumerate}
\item $t_n\in 2^{\sum_{i\leq n} h(n)} $, and
\item $\forall f\in F$ $\exists n\in \Nset $ $t_n\subs x_f$.
\end{enumerate}
Let $g\in\Pset$ be such that $g(n)\leq 2^{h(n)}$ for all $n\in\Nset$,
and $g(n)=m$ whenever $t_{n+1}=t_n{}^{\smallfrown} s^{h(n)}_m$. Note that $g(n)=f(n)$
whenever $t_n\subs x_f$.
 \end{proof}

In particular, the theorem evaluates $\non(\smz)$ for two groups: the compact
boolean group $\Cset$ and the Baer-Specker group $\mathbb Z^\Nset$.
In order to extend these to a wider class of groups we shall need  two easy observations:

As mentioned before,  strong measure zero is  a uniform property, in particular,
a uniformly continuous image of a $\smz$ set is $\smz$ and, on the other hand, if $X$
uniformly embeds into $Y$, then any set $A\subs X$ that is not $\smz{}$
in $X$ is not $\smz{}$ in $Y$ either. It follows that:

\begin{lem}\label{lem:Best}
\begin{enum}
\item
If a space $Y$ is a uniformly continuous image of a space $X$ then $\nonS{X}\leq\nonS{Y}$.
\item
If $X$ uniformly embeds into $Y$, then $\nonS{X}\geq\nonS{Y}$.
\end{enum}
\end{lem}

\begin{lem}\label{hur}
A $\mathsf{CLI}$ Polish group is either
locally compact, or  contains a uniform copy of $\Pset$.
\end{lem}

\begin{proof}
Let $\grG$ be a group equipped with a complete, left-invariant metric $d$.
Assuming  $\grG$ is not locally compact no open set is totally bounded,
hence for every $\eps>0$ exists $\del>0$ such that the ball $B(1,\eps)$
contains an infinite set of points that are pairwise at least $\del$ apart.
Using  this fact construct, for each $n\in\Nset$, $\eps_n>0$ and an infinite
set $\{x_n^i:i\in\Nset\}\subs B(1,\eps_n)$ such that
if $i\neq j$ then $d(x_n^i,x_n^j)>5\eps_{n+1}$.
For $s\in\Nset^n$ let
$y_s=x_0^{s(0)}\cdot x_1^{s(1)}\cdot x_{n-2}^{s(n-2)}\cdot\dots\cdot x_{n-1}^{s(n-1)}$.
The construction ensures that for any $f\in\Pset$ the sequence
$\seq{y_{f\rest n}}$ is Cauchy. Let $z_f$ be its limit.
It is easy to check that since $d$ is left-invariant,
the mapping $f\mapsto z_f$ is a uniform embedding of $\Pset$ into $\grG$.
\end{proof}

\begin{thm}
Let $\grG$ be a $\mathsf{CLI}$ Polish group.
\begin{enum}
\item{} If $\grG$ is locally compact, then $\nonS\grG=\eq$,
\item{} if $\grG$ is not locally compact, then $\nonS\grG=\covM$.
\end{enum}
\end{thm}

\begin{proof} As every Polish group contains a (uniform) copy of $\Cset$,
$$\covM\leq \nonS\grG\leq \eq$$
for every Polish group $\grG$. Similarly, $\nonS\grG\leq\covM$  whenever
$\grG$ contains a uniform copy of $\Pset$ by Theorem \ref{thm:Roth}.
In particular, $\nonS\grG = \covM$ if $\grG$ is a
non-locally compact $\mathsf{CLI}$  group.

\smallskip

So all that remains to be seen is that $\eq\leq \nonS\grG$ for any locally compact
Polish group $\grG$.
Write $\grG$ as an increasing union of compact subsets $K_n$, $n\in\Nset$.
As each $K_n$ is a uniformly continuous image of
$\Cset$ (see e.g.\ \cite[Theorem 4.18]{MR1321597}), every subset of $K_n$
which is not of strong measure zero has size at least $\eq$ by  Lemma~\ref{lem:Best}(i)
and Theorem~\ref{thm:Roth}(iii). On the other hand, as $\smz(\grG)$ is a \si ideal
(Proposition \ref{sigma}), every subset of $\grG$ which is not $\smz$ has
a non-$\smz$ intersection with one of the $k_n$'s, hence
$\eq\leq \nonS\grG$.
\end{proof}

It is, of course, a natural question whether the result of the Theorem
(or of the preceding lemma) remains true also for
Polish groups which are not  $\mathsf{CLI}$.

\bigskip

The final remark of  this section deals with \emph{transitive covering for category}
(considered by
Bartoszy\'nski~\cite[2.7]{MR1350295} for $\Cset$, and Miller and
Stepr\={a}ns~\cite{MR2224048} for general Polish group and further studied in \cite{MR3453581}):
$$
  \cov^*(\MM(\grG))=\min\{|A|: \ A\subs\grG \text{ and }
  A\cdot M=\grG \text{ for some meager set } M\subs \grG\}
$$

By Theorem \ref{kysiak}, for a locally compact group $\grG$ strong measure zero
sets coincide with the sets whose meager translates do not cover $\grG$, hence,
in particular, $\nonS \grG =\cov^*(\MM(\grG))$ for every locally compact group $\grG$.
It follows from Prikry's Proposition \ref{prop:triv} that
$\cov^*(\MM(\grG))\leq \nonS \grG $, for every Polish group $\grG$, hence
$\covM\leq\cov^*(\MM(\grG))\leq\eq$ for any Polish group. As $\nonS \grG=\covM$
for all $\mathsf{CLI}$  groups which are not locally compact, we can conclude that:

\begin{coro}
$\nonS \grG =\cov^*(\MM(\grG))$ for any $\mathsf{CLI}$  group.
\end{coro}

The conjecture is that the two numbers coincide for any Polish group.

\begin{question}
Is $\cov^*(\MM(\grG))=\covM$ for all Polish groups which are not locally compact?
\end{question}

%{\bf Either delete $\ed$ or make use of it... meager additive?}

\section{Strong measure zero and Hausdorff measures}
\label{sec:haus}
%%%%%%%%%%%%%%%%%%%%%%%%%%%%%%%%%%%%%%%%%%%%%%%%%%%%%%%%%%%%%%%%%
%%%%%%%%%%%%%%%%%%%%%%%%%%%%%%%%%%%%%%%%%%%%%%%%%%%%%%%%%%%%%%%%%

As mentioned above, there is a profound connection between strong measure zero and
Hausdorff measures.
The following characterizations of strong measure zero
in terms of Hausdorff measures and dimensions were proved
in \cite{Zin_M-add}. They are based on a classical
Besicovitch result~\cite{MR1555386,MR1555389}.

%%%%%%%%%%%%%%%%%%%%%%%%%%%%%%%%%%%%%%%%%%%%%%%%%%%%%%%%%%%%%%%%%
\subsection*{Hausdorff measure}
%%%%%%%%%%%%%%%%%%%%%%%%%%%%%%%%%%%%%%%%%%%%%%%%%%%%%%%%%%%%%%%%%

Before getting any further we need to review Hausdorff measure
and dimension. We set up the necessary definitions
and recall relevant facts.

%Recall that if As usual, $X$ denotes a separable metric space and $\diam A$ denotes the diameter of
%a set $A\subs X$.
%A closed ball of radius $r$ centered at $x$ is denoted by $B(x,r)$.

A non-decreasing, right-continuous function $h:[0,\infty)\to[0,\infty)$
such that $h(0)=0$ and $h(r)>0$ if $r>0$ is called a \emph{gauge}.
Gauges are often ordered as follows, cf.~\cite{MR0281862}:
$$
  g\prec h\quad \overset{\mathrm{def}}{\equiv}
  \quad\lim_{r\to0+}\frac{h(r)}{g(r)}=0.
$$
%In the case when $h(r)=r^s$ for some $s>0$, we write $g\prec s$ instead of
%$g\prec h$.

Notice that for any sequence $\seq{h_n}$ of gauges there is
a gauge $h$ such that $h\prec h_n$ for all $n$.

Given $\del>0$, call a cover $\mc A$ of a set $E\subs X$ a
\emph{$\del$-fine cover} if $\forall A\in\mc A\ \diam A\leq\del$.
If $h$ is a gauge,
the \emph{$h$-dimensional Hausdorff measure} $\hm^h(E)$ of
a set $E\subs X$ is defined thus:
For each $\del>0$ define
$$
  \hm^h_\delta(E)=
  \inf\left\{\sum_{n\in\Nset}h(\diam E_n):
  \text{$\{E_n\}$ is a countable $\delta$-fine cover of $E$}\right\}
$$
and let
$$
  \hm^h(E)=\sup_{\delta>0}\hm^h_\delta(E).
$$

In the common case of $h(r)=r^s$ for some $s>0$, we write $\hm^s$ for
$\hm^h$, and likewise for $\hm^h_\delta$.

Properties of Hausdorff measures are well-known. The following, including
the two lemmas, can be found e.g.~in~\cite{MR0281862}.
The restriction of $\hm^h$ to Borel sets is a $G_\del$-regular Borel measure.
Recall that a sequence of sets $\seq{E_n}$ is termed a \emph{$\lambda$-cover}
of $E\subs X$ if every point of $E$ is contained in infinitely many $E_n$'s.
%%
%\begin{lem}\label{lambda}
%$\hm^h(E)=0$ if and only if $E$ admits a countable
%$\lambda$-cover $\seq{E_n}$ such that $\sum_{n\in\Nset}h(\diam E_n)<\infty$.
%\end{lem}

\begin{thm}[\cite{MR1555386,MR1555389}]\label{besic}
A metric space $X$ is \smz{} if and only if
$\hm^h(X)=0$ for each gauge $h$.
\end{thm}
\begin{proof}
Suppose first that $X$ is \smz.
Let $h$ be a gauge. For each $\del>0$ pick a sequence $\sseq{\eps_n}$ such that
$0<\eps_n<\del$ and
$h(\eps_n)<\del2^{-n}$. There is a cover $\{U_n\}$ of $X$ such that
$\diam U_n<\eps_n$ for all $n$. Obviously $\sum h(\diam U_n)<\del$ and it
follows that $\hm^h_\del(X)\leq\del$. Let $\del\to0$ to get $\hm^h(X)=0$.

Now suppose that $\hm^h(X)=0$ for every gauge.
Let $\sseq{\eps_n}$ be a sequence of positive numbers.
Choose a gauge $h$ such that $h(\eps_n)>\frac1n$.
Since $\hm^h(X)=0$, there is a countable cover $\{U_n\}$ such that $\sum h(\diam U_n)<1$.
As $h$ is right-continuous, there are $\del_n>\diam U_n$ such that
$\sum h(\del_n)<1$. Since $\del_n>0$,
rearranging the sequence we may suppose that $\del_n$ decrease. Therefore
$n h(\del_n)\leq \sum_{i<n} h(\del_n)<1$. It follows that
$h(\del_n)<\frac1n<h(\eps_n)$ and consequently $\del_n<\eps_n$
and in particular $\diam U_n<\eps_n$, as required.
\end{proof}

%%
%\begin{lem}\label{lemHaus}
%\begin{enum}
%\item If $\hm^h(X)<\infty$ and $h\prec g$, then $\hm^g(X)=0$.
%\item If $\hm^h(X)=0$, then there is $g\prec h$ such that $\hm^g(X)=0$.
%\end{enum}
%\end{lem}
%%
We will need a cartesian product inequality.
Given two metric spaces $X$ and $Y$ with
respective metrics $d_X$ and $d_Y$, provide the cartesian product $X\times Y$
with the maximum metric
\begin{equation}\label{maxmetric}
  d\bigl((x_1,y_1),(x_2,y_2)\bigr)=\max(d_X(x_1,x_2),d_Y(y_1,y_2)).
\end{equation}
A gauge $h$ satisfies the \emph{doubling condition} or $h$ is \emph{doubling}
if $\varlimsup_{r\to0}\frac{h(2r)}{h(r)}<\infty$.
\begin{lem}[{\cite{MR0318427,MR1362951}}]\label{howroyd}
Let $X,Y$ be metric spaces, $g$ a gauge and $h$ a doubling gauge. Then
$\hm^h(X)\,\hm^g(Y)\leq\hm^{hg}(X\times Y)$.
\end{lem}

The following lemma on uniformly
continuous mappings is well-known, see, e.g., \cite[Theorem 29]{MR0281862}.
%
%\begin{lem}\label{lipschitz}%\label{lipschitz2}
%Let $f:(X,d_X)\to (Y,d_Y)$ be a mapping.
%\begin{enum}
%\item
%If $f$ is uniformly continuous and a gauge $g$ is its modulus, i.e.,
%%
%\begin{equation}\label{lip2}
%  d_Y(f(x),f(y))\leq g(d_X(x,y)),\quad x,y\in X,
%\end{equation}
%%
%then $\hm^h(f(X))\leq \hm^{h{\circ}g}(X)$ for any gauge $h$.
%\item
%If $f$ is Lipschitz with Lipschitz constant $L$,
%then $\hm^s(f(X))\leq L^s\hm^s(X)$ for any $s>0$.
%\end{enum}
%\end{lem}
\begin{lem}\label{lipschitz}%\label{lipschitz2}
Let $f:(X,d_X)\to (Y,d_Y)$ be a uniformly continuous mapping
and $g$ its modulus, i.e., $d_Y(f(x),f(y))\leq g(d_X(x,y))$ for all $x,y\in X$.
%
%\begin{equation}\label{lip2}
%  d_Y(f(x),f(y))\leq g(d_X(x,y)),\quad x,y\in X.
%\end{equation}
%
Then $\hm^h(f(X))\leq \hm^{h{\circ}g}(X)$ for any gauge $h$.
\end{lem}
Recall that the \emph{Hausdorff dimension} of $X$ is defined by
$$
  \hdim X=\sup\{s>0:\hm^s(X)=\infty\}=\inf\{s>0:\hm^s(X)=0\}.
$$
Properties of Hausdorff dimension are well-known. In particular,
$\hdim X=0$ if $X$ is countable; and if $f:X\to Y$ is Lipschitz, then
$\hdim f(X)\leq\hdim X$.

\begin{coro}[\cite{Zin_M-add}]\label{basicHnull}
Let $X$ be a metric space. The following are equivalent.
\begin{enum}
\item $X$ is \smz{},
\item $\hdim f(X)=0$ for each uniformly continuous mapping $f$ on $X$,
\item $\hdim(X,\rho)=0$ for each uniformly equivalent metric $\rho$ on $X$.
\end{enum}
\end{coro}
\begin{proof}
(i)\Implies(ii)
Let $s>0$ be arbitrary. Let $f:X\to Y$ be uniformly continuous and let $g$
be the modulus of $f$. Define $h(x)=(g(x))^s$.
Lemma~\ref{lipschitz}(i) yields $\hm^s(f(X))\leq\hm^h(X)$.
By the above theorem $\hm^h(X)=0$ and thus $\hm^s(f(X))=0$.
Since this holds for all $s>0$, it follows that $\hdim f(X)=0$.

(ii)\Implies(iii) is trivial.

(iii)\Implies(i)
Denote by $d$ the metric of $X$.
Let $h$ be a gauge. Choose a strictly increasing, convex (and in particular subadditive)
gauge $g$ such that $g\prec h$.
The properties of $g$ ensure that $\rho(x,y)=g(d(x,y))$
is a uniformly equivalent metric on $X$. The identity
map $\id_X:(X,\rho)\to(X,d)$ is of course uniformly continuous and its modulus
is $g^{-1}$, the inverse of $g$. Hence by Lemma~\ref{lipschitz}(i)
$\hm^h(X,d)\leq\hm^{h\circ g^{-1}}(X,\rho)$.
Since $g\prec h$, we have $\hm^{h\circ g^{-1}}(X,\rho)\leq\hm^1(X,\rho)$
and $\hm^1(X,\rho)=0$ by (ii). Thus $\hm^h(X,d)=0$.
\end{proof}

Our next goal is to characterize \smz{} by behavior of cartesian products. Recall that for $p\in\CCset$, $\cyl p=\{x\in\Cset:p\subs x\}$ denotes the cone determined
by $p$ and for $T\subs\CCset$ we let $\cyl T=\bigcup_{p\in T}\cyl p$.

The coordinatewise addition modulo $2$ makes $\Cset$ a compact topological group.
Routine proofs show that in the metric of the least difference
(defined in the introduction) the $1$-dimensional Hausdorff measure
$\hm^1$ coincides on Borel sets with its Haar measure, i.e.,
the usual product measure on $\Cset$. In particular $\hm^1(\Cset)=1$.

We consider the important \si ideal $\EE$ on $\Cset$ generated by closed null sets, i.e.,
the ideal of all subsets of $\Cset$ that are contained in an $F_\sigma$ set of Haar measure zero.
\begin{lem}\label{EC}
\begin{enum}
%\item $E\in\EE$ if and only if there is $h\prec 1$ such that $E\in\NNs(\uhm^h_0)$,
\item For each $I\in[\Nset]^\Nset$, the set $C_I=\{x\in\Cset:x\rest I\equiv0\}$
is in $\EE$.
\item For each $h\prec 1$ there is $I\in[\Nset]^\Nset$ such that $\hm^h(C_I)>0$.
\end{enum}
\end{lem}
\begin{proof}
(i) Let $I\in[\Nset]^\Nset$. For each $n\in\Nset$,
the family $\{\cyl p:p\in C_I\rest n\}$ is obviously a $2^{-n}$-fine cover
of $C_I$ of cardinality $2^{\abs{n\setminus I}}$. Therefore
$\hm^1_{2^{-n}}(C_I)\leq2^{\abs{n\setminus I}}2^{-n} =2^{-\abs{n\cap I}}$.
Hence $\hm^1(C_I)\leq\lim_{n\to\infty}2^{-\abs{n\cap I}}=0$.

(ii)
$h\prec 1$ yields  $\frac{h(2^{-n})}{2^{-n}}\to\infty$. Therefore there is
$I\in[\Nset]^\Nset$ sparse enough to satisfy
$2^{\abs{n\cap I}}\leq\frac{h(2^{-n})}{2^{-n}}$,
i.e., $2^{-\abs{n\setminus I}}\leq h(2^{-n})$ for all $n\in\Nset$.
Consider the product measure $\lambda$ on $C_I$ given as follows: If $p\in2^n$ and
$\cyl p\cap C_I\neq\emptyset$, put
$\lambda(\cyl p\cap C_I)=2^{-\abs{n\setminus I}}$.
Straightforward calculation shows that $h(\diam E)\geq\lambda(E)$ for each $E\subs C_I$.
Hence $\sum_n h(\diam E_n)\geq\sum_n\lambda(E_n)\geq\lambda(C_I)=1$
for each cover $\{E_n\}$ of $C_I$ and $\hm^h(C_I)\geq 1$ follows.
\end{proof}

\begin{thm}\label{prodHnull}
The following are equivalent.
\begin{enum}
\item $X$ is \smz{},
%\item $\hm^h(X\times Y)=0$ whenever $h\in\HH$ and $\uhm_0^h(Y)=0$,
\item $\hm^h(X\times Y)=0$ for every gauge $h$ and every compact metric space $Y$
such~that $\hm^h(Y)=0$,
\item $\hm^1(X\times E)=0$ for every $E\in\EE$.
%\item $\hm^1(X\times C_I)=0$ for every $I\in[\Nset]^\Nset$.
\end{enum}
\end{thm}
\begin{proof}
(i)\Implies(ii):
Suppose $X$ is \smz{} and let $Y$ be compact.
Let $\del>0$. Since $\hm^h(Y)=0$, for each $j\in\Nset$
there is a finite cover $\mc U_j$ of $Y$ such that
$\sum_{U\in\mc U_j}h(\diam U)<2^{-j-1}\del$.
We may also assume that $\diam U<\del$ for all $U\in\mc U_j$.

Let $\eps_j=\min\{\diam U:U\in\mc U_j\}$. Since $X$ is \smz,
there is a cover $\{V_j\}$ of $X$ such that $\diam V_j\leq\eps_j$. Define
$$
  \mc W=\{V_j\times U:j\in\Nset,\,U\in\mc U_j\}.
$$
$\mc W$ is obviously a cover of $X\times Y$. The choice of $\eps_j$ yields
$\diam(V_j\times U)=\diam U$ for all $j$ and $U\in\mc U_j$. Therefore
$$
  \sum_{W\in\mc W}h(\diam W)=
  \sum_{j\in\Nset}\sum_{U\in\mc U_j}h(\diam U)<
  \sum_{j\in\Nset}2^{-j-1}\del=\del.
$$
It follows that $\hm_\del^h(X\times Y)<\del$, and
$\hm^h(X\times Y)=0$ obtains by letting $\del\to0$.

(ii)\Implies(iii)\Implies(iv) is trivial.

(iv)\Implies(i):
Suppose $X$ is not \smz{}. We will show that $\hm^1(X\times E)>0$
for some $E\in\EE$.
By assumption and Theorem~\ref{basicHnull} there is a gauge $h$ such that $\hm^h(X)>0$.
We may assume $h$ be concave and $h(r)\geq\sqrt r$.
In particular, by concavity of $h$ the function $g(r)=r/h(r)$ is increasing.
Moreover, $h(r)\geq\sqrt r$ yields $\lim_{r\to0}g(r)=0$, i.e., $g$ is a gauge, and
$g\prec1$. Further, $g(2r)=2r/h(2r)\leq2r/h(r)=2g(r)$, i.e., $g$ is doubling.

Use Lemma~\ref{EC}(ii) to find $I\in[\Nset]^\Nset$ such that $\hm^g(C_I)>0$
and let $E=C_I$. By Lemma~\ref{EC}(i), $E\in\EE$.
Since $g$ is doubling, Lemma~\ref{howroyd} applies:
\begin{equation*}
  \hm^1(X\times C_I)=\hm^{h\cdot g}(X\times C_I)\geq\hm^h(X)\cdot\hm^g(C_I)>0.
  \qedhere
\end{equation*}
\end{proof}
\begin{coro}\label{hdimx}
If $X$ is \smz{}, then $\hdim X\times Y=\hdim Y$ for every compact metric space $Y$.
\end{coro}
Note that this consistently fails when we drop the assumption that $Y$ is compact:
By a classical example (cf.~\cite[534P]{fremlin5}), if $\cov(\MM)=\co$,
then there is a \smz{} set $X\subs\Rset$ such that $X+X=\Rset$.
Since $X+X$ is a Lipschitz image of $X\times X$, we have
$\hdim X\times X\geq\hdim X+X=1$, while $X$ is \smz{} and $\hdim X=0$.

%%%%%%%%%%%%%%%%%%%%%%%%%%%%%%%%%%%%%%%%%%%%%%%%%%%%%%%%%%%%%%%%%
%%%%%%%%%%%%%%%%%%%%%%%%%%%%%%%%%%%%%%%%%%%%%%%%%%%%%%%%%%%%%%%%%
\section{Sharp measure zero}
%%%%%%%%%%%%%%%%%%%%%%%%%%%%%%%%%%%%%%%%%%%%%%%%%%%%%%%%%%%%%%%%%
%%%%%%%%%%%%%%%%%%%%%%%%%%%%%%%%%%%%%%%%%%%%%%%%%%%%%%%%%%%%%%%%%

Recall that a set $S$ in a Polish group is called \emph{meager-additive}
if $S\cdot M$ is meager for every meager set $M$. This is obviously
a strengthening of the ``algebraic'' characterization of \smz{}
in the Galvin-Mycielski-Solovay Theorem. Meager-additive sets, in
particular in $\Cset$, have received a lot of attention, e.g.,
in \cite{MR1324470,MR1350295,MR1610427,MR2545840}.

Very recently it was shown that meager-additive sets are characterized by a
combinatorial condition very similar to the definition of \smz{} and
also in terms of Hausdorff measures. In this section we will have a look at
these descriptions.

\begin{defn}
A set $S\subs X$ in a complete metric space $X$ has \emph{sharp measure zero}
if for every gauge $h$ there is a \si compact set $K\sups S$ such that
$\hm^h(K)=0$.
\end{defn}

We first work towards an intrinsic definition equivalent to the one above.
The following variation of Hausdorff measure seems to be the right notion for that.
Let $h$ be a gauge. For each $\del>0$ define
$$
  \uhm^h_\delta(E)=
  \inf\left\{\sum_{n=0}^N h(\diam E_n):
  \text{$\{E_n:n\leq N\}$ is a \emph{finite} $\delta$-fine cover of $E$}\right\}
$$
and let
$$
  \uhm_0^h(E)=\sup_{\delta>0}\uhm^h_\delta(E).
$$
Note the striking similarity with the Hausdorff measure.
The only difference is that only finite covers are taken into account.
It is easy to check that $\uhm_0^h$ is finitely subadditive. However,
it is not a measure, since it may (due to the finite covers) lack \si additivity.
To turn it into a measure we need to apply the operation
known as Munroe's \emph{Method I construction} (cf.~\cite{MR0053186} or \cite{MR0281862}):
$$
  \uhm^h(E)=\inf\Bigl\{\sum_{n\in\Nset}\uhm^h_0(E_n):
  E\subs\bigcup_{n\in\Nset}E_n\Bigr\}.
$$
Thus the defined set function $\uhm^h$ is indeed an outer measure whose restriction to
Borel sets is a Borel measure.
We will called it \emph{$h$-dimensional upper Hausdorff measure}.

Upper Hausdorff measures behave much like Hausdorff measures.
We list some important properties of upper Hausdorff measures.
We refer to \cite{Zin_M-add} for details.

Denote $\NNs(\uhm_0^h)$ the smallest \si additive ideal that contains
all sets $E$ with $\uhm_0^h(E)=0$. Note that while $E\in\NNs(\uhm_0^h)$
\Implies $\uhm^h(E)=0$, the reverse implication in general fails.
Write $E_n\upto E$ to denote that $\seq{E_n}$ is an increasing sequence
of sets with union $E$.  The following lists some basic features
of $\uhm^h$ and $\uhm_0^h$.

\begin{lem}[{\cite{Zin_M-add}}]\label{lem1}
Let $h$ be a gauge and $E$ a set in a metric space.
\begin{enum}
\item If $\uhm_0^h(E)<\infty$, then $E$ is totally bounded.
\item $\uhm_0^h(E)=\uhm_0^h(\clos E)$.
\item $\uhm_0^h(E)=\hm^h(E)$ if $E$ is compact.
\item If $E\in\NNs(\uhm_0^h)$, then $\uhm^h(E)=0$.
\item If $X$ is complete, $E\subs X$ and $E\in\NNs(\uhm_0^h)$, then there is a \si compact set
$K\sups E$ such that $\hm^h(K)=0$.
%\item $\uhm^h(E)=(\uhm_0^h)^\tto(E)$.
\item If $X$ is complete and $E\subs X$, then
$\uhm^h(E)=\inf\{\hm^h(K):\text{$K\sups E$ is \si compact}\}$.
\item In particular $\uhm^h(E)=\hm^h(E)$ if $E$ is \si compact.
\item If $g\prec h$ and $\uhm^g(E)<\infty$, then $E\in\NNs(\uhm_0^h)$; in particular,
$\uhm^h(E)=0$.
%\item If $E\in\NNs(\uhm_0^h)$, then there is a sequence $E_n\upto E$ such that
%$\uhm_0^h(E_n)=0$ for all $n$.
%\item If $\uhm^h(E)<s$, then there is a sequence $E_n\upto E$ such that
%$\sup\uhm_0^h(E_n)<s$.
\end{enum}
\end{lem}

\begin{lem}[{\cite{Zin_M-add}}]\label{gammagr}
$E\in\NNs(\uhm_0^h)$ if and only if $E$ has a $\gamma$-groupable
cover $\seq{U_n}$ such that $\sum_{n\in\Nset}h(\diam U_n)<\infty$.
\end{lem}
\begin{proof}
Suppose that $E\in\NNs(\uhm_0^h)$. Let $E_n\upto E$ be such that $\uhm^h_0(E_n)=0$.
For each $n$ let $\mc G_n$ be a finite cover of $E_n$ such that
$\sum_{G\in\mc G_n}h(\diam G)<2^{-n}$. Since the family $\mc G=\bigcup_n\mc G_n$
is obviously a $\gamma$-groupable cover, we are done.

In the opposite direction, suppose that $\seq{U_n}$ is a $\gamma$-groupable
cover $\seq{U_n}$ such that $\sum_{n\in\Nset}h(\diam U_n)<\infty$
with witnessing families $\mc G_j$.
Let $E_k=\bigcap_{j\geq k}\bigcup\mc G_j$. Then $E=\bigcup_{k\in\Nset}E_k$.
For each $k$, the set $E_k$ is covered
by each $\mc G_j$, $j\geq k$, and $\sum_{G\in\mc G_j}h(\diam G)$ is as
small as needed for $j$ large enough. Hence $\uhm^h_0(E_k)=0$
and consequently $E\in\NNs(\uhm_0^h)$.
\end{proof}

It is straightforward from Lemma~\ref{lem1} that
the following intrinsic definition of sharp measure zero is consistent with the one above.

\begin{defn}
A metric space $X$ has sharp measure zero if $\uhm^h(X)=0$ for every gauge $h$.
Sharp measure zero is abbreviated as \ssmz{}.
\end{defn}

It is no surprise that Theorem~\ref{basicHnull} has a counterpart for \ssmz,
with basically the same proof. Upper Hausdorff dimension is defined as expected:
$$
  \uhdim X=\sup\{s>0:\uhm^s(X)=\infty\}=\inf\{s>0:\uhm^s(X)=0\},
$$
see~\cite{Zin_M-add,MR2957686} for more on the Hausdorff dimension.

\begin{thm}[\cite{Zin_M-add}]\label{basicUhnull}
Let $X$ be a metric space. The following are equivalent.
\begin{enum}
\item $X$ is \ssmz{},
\item $\uhdim f(X)=0$ for each uniformly continuous mapping $f$ on $X$,
\item $\uhdim(X,\rho)=0$ for each uniformly equivalent metric $\rho$ on $X$.
\end{enum}
\end{thm}

It is straightforward from the definition and Theorem~\ref{basicUhnull}
that \ssmz{} is a \si additive property and that it is
preserved by uniformly continuous mappings.

Sharp measure zero can be described in terms of covers. The description is strikingly
similar to the Borel's definition of strong measure zero.

A countable cover $\{U_j\}$ of $X$ is a called
a \emph{$\gamma$-cover} if each $x\in X$ belongs to all but finitely
many $U_j$.

The following notion was studied, e.g., in \cite{MR2029229}.
A sequence $\sseq{W_n}$ of sets in $X$ is called a \emph{$\gamma$-groupable cover}
if there is a partition $\Nset=I_0\cup I_1\cup I_2\cup\dots$
into consecutive finite intervals (i.e.~$I_{j+1}$ is on the right of $I_j$ for all $j$)
such that the sequence $\sseq{\bigcup_{n\in I_j}W_n:j\in\Nset}$ is a $\gamma$-cover.
The partition $\sseq{I_j}$ will be occasionally called a \emph{witnessing partition}
and the finite families $\{U_n:n\in I_j\}$
will be occasionally called \emph{witnessing families}.

\begin{thm}[\cite{Zin_M-add}]\label{besic2}
A metric space $X$ is \ssmz{} if and only if it has the following property:
for every sequence $\seq{\eps_n}$
of positive real numbers there is a $\gamma$-groupable cover $\{U_n:n\in\Nset\}$
of $X$ such that $\diam U_n \leq \eps_n$ for all $n$.
\end{thm}
\begin{proof}
The pattern of the proof is the same as that of the proof of the Besicovitch's
theorem~\ref{besic}, but there are details that make it much more involved.

The forward implication is easy.
Let $h$ be a gauge. Pick $\eps_n>0$ such that
$\sum_nh(\eps_n)<\infty$. By assumption, there is a
$\gamma$-groupable cover $\sseq{G_n}$ such that $\diam G_n\leq\eps_n$.
Therefore $\sum_nh(\diam G_n)\leq\sum_nh(\eps_n)<\infty$.
Now apply Lemma~\ref{gammagr} to conclude that $X\in\NNs(\uhm_0^h)$
and in particular $\uhm^h(X)=0$.

The reverse implication:
Let $\sseq{\eps_n}\in(0,\infty)^\Nset$.
Choose a gauge $g$ such that
$g(\eps_n)>\frac1n$ for all $n\geq1$ and then a gauge $h\prec g$.
Since $\uhm^h(X)=0$, Lemma~\ref{lem1}(viii) yields $X\in\NNs(\uhm_0^g)$,
which in turn yields, with the aid of Lemma~\ref{gammagr}, a
$\gamma$-groupable cover $\seq{G_n}$ such that $\sum_ng(\diam G_n)<\infty$.
Let $\{I_j:j\in\Nset\}$ be the witnessing partition and $\mc G_j=\{G_n:n\in I_j\}$
the witnessing families.

We want to permute the cover so that diameters decrease.
Some of the diameters may be $0$.
Also, the permutation may break down the witnessing families.
We thus have to exercise some care.

For each $n$ choose $\del_n>\diam G_n$ so that $\sum_ng(\del_n)<\infty$.
Next choose an increasing sequence $\seq{j_k}$ satisfying
for all $k\in\Nset$
\begin{enumerate}
\item[(a)] $\sum\{g(\del_n):n\in I_{j_k}\}<2^{-k-1}$,
\item[(b)] $\max\{\del_n:n\in I_{j_{k+1}}\}<\min\{\del_n:n\in I_{j_{k}}\}$.
\end{enumerate}
Let $I=\bigcup_{k\in\Nset}I_{j_k}$.
Rearrange $G_n$'s within each group $\mc G_{j_k}$ so that $\sseq{\del_n:n\in I_{j_k}}$
form a non-increasing sequence. Together with (b) this ensures that
the sequence $\sseq{\del_n:n\in I}$ is non-increasing.

For each $n\in\Nset$ let $\widehat n\in I$ be the unique index such that
$n=\abs{I\cap \widehat n}$ and define $H_n=G_{\widehat n}$.
It follows, with the aid of (a)
and the definition of $g$, that for all $n\in\Nset$
\begin{align*}
  g(\diam H_n)=g(\diam G_{\widehat n})&\leq g(\del_{\widehat n})
  \leq\frac{1}{n}\sum\{g(\del_m):m\in I,m\leq \widehat n\}\\
  &\leq\frac{1}{n}\sum\{g(\del_m):m\in I\}
  \leq\frac{1}{n}<g(\eps_n)
\end{align*}
and thus $\diam H_n\leq\eps_n$ for all $n$.
Moreover, the families $\mc G_{j_k}$, $k\in\Nset$, witness that $\seq{H_n}$
is a $\gamma$-groupable cover.
\end{proof}

\begin{thm}[\cite{Zin_M-add}]\label{prodUhnull}
Let $X$ be a metric space.
The following are equivalent.
\begin{enum}
\item $X$ is \ssmz,
%\item for each gauge $h$, $Y\in\NNs(\uhm_0^h)$ and each complete space $Z\sups X$
%there is a \si compact $F$, $X\subs F\subs Z$,
%such that $\uhm^h(F\times Y)=0$,
\item $\uhm^h(X\times Y)=0$ for each gauge $h$ and $Y$ such that $\uhm_0^h(Y)=0$,
\item $\uhm^1(X\times E)=0$ for each $E\in\EE$.
%\item $\uhm^1(X\times C_I)=0$ for each $I\in[\Nset]^\Nset$.
\end{enum}
\end{thm}
\begin{proof}
(i)\Implies(ii):
Suppose $X$ is \ssmz{}. Let $h$ be a gauge and $\uhm_0^h(Y)=0$.
By Lemma~\ref{gammagr} there is a $\gamma$-groupable cover $\mc U$ of $Y$ such that
$\sum_{U\in\mc U}h(\diam U)<\infty$. For each $U\in\mc U$ there is $\del_U>\diam U$
such that $\sum_{U\in\mc U}h(\del_U)<\infty$.
Denote by $\mc U_j$ the witnessing families and
let $\eps_j=\min\{\del_U:U\in\mc U_j\}$.
Since $X$ is \ssmz, Theorem~\ref{besic2} yields a $\gamma$-groupable cover
$\seq{V_j}$ of $X$ such that $\diam V_j\leq\eps_j$.
%We may assume that each $V_j$ is a closed subset of $Z$.
Denote by $\mc V_k$ the witnessing families.
Define a family of sets in $X\times Y$
$$
  \mc W=\{V_j\times U:j\in\Nset,\,U\in\mc U_j\}.
$$
It is routine to check that $\mc W$ is a $\gamma$-groupable cover of $F\times Y$.
Since $\diam(V_j\times U)\leq\del_U$ for all $j$ and $U\in\mc U_j$ by the choice
of $\eps_j$, we have
$$
  \sum_{W\in\mc W}h(\diam W)\leq
  \sum_{U\in\mc U}h(\del_U)<\infty.
$$
Thus it follows from Lemma~\ref{gammagr} that $X\times Y\in\NNs(\uhm_0^h)$ and
in particular $\uhm^h(X\times Y)=0$.

(ii)\Implies(iii) is trivial.
The proof of (iii)\Implies(i) is very much like that of Theorem~\ref{prodHnull}.
Suppose $X$ is not \ssmz. We need to find $E\in\EE$ such that $\uhm^1(X\times E)>0$.
By assumption there is a gauge $h$ such that $\uhm^h(X)>0$.
We may suppose $h$ is concave, and
find a doubling gauge $g\prec1$ such that $g(r)h(r)=r$.
Then use Lemma~\ref{EC}(ii) to find $I\in[\Nset]^\Nset$ such that $\hm^g(C_I)>0$.
We now need a product inequality on upper Hausdorff measures
analogous to Lemma~\ref{howroyd} proved in~\cite[3.5,7.4]{Zin_M-add}. 
\begin{lems}[\cite{Zin_M-add}]\label{uhowroyd}
Let $X,Y$ be metric spaces and $g$ a gauge and $h$ a doubling gauge. Then
$\hm^h(X)\,\uhm^g(Y)\leq\uhm^{hg}(X\times Y)$.
\end{lems}

\noindent
Using this lemma, we get
\begin{equation*}
  \uhm^1(X\times C_I)=\uhm^{h\cdot g}(X\times C_I)\geq\uhm^h(X)\cdot\hm^g(C_I)>0.
  \qedhere
\end{equation*}
\end{proof}

As we already mentioned, under $\cov(\MM)=\co$ there is an example
(\cite[534P]{fremlin5}) of a \smz{} set $X\subs\Rset$ such that
$X\times X$ is not \smz{}. Scheepers~\cite{MR1779763} examines thoroughly
conditions imposed on a \smz{} set $X$ that would ensure that a product of $X$
with another \smz{} set is \smz.
A recent roofing result claims that if one of the factors is \ssmz{},
then the product is \smz.
\begin{thm}[\cite{Zin_M-add}]\label{productHUH}
\begin{enum}
\item If $X$ and $Y$ are \ssmz{}, then $X\times Y$ is \ssmz{}.
\item If $X$ is \smz{} and $Y$ is \ssmz{}, then $X\times Y$ is \smz{}.
\end{enum}
\end{thm}
\begin{proof}
Suppose $Y$ is \ssmz. By Lemma~\ref{lem1}(viii),
$Y\in\NNs(\uhm_0^h)$ for all gauges $h$.

(i) If $X$ is \ssmz{}, then Theorem~\ref{prodUhnull}(ii) yields
$\uhm^h(X\times Y)=0$ for all gauges $h$.

(ii) Let $h$ be a gauge. Since $Y$ is \ssmz, it is \si totally bounded and
therefore there is a \si compact set $K\sups Y$
in the completion of $Y$ such that $\hm^h(K)=0$.
Since $X$ is \smz{}, Theorem~\ref{prodHnull}(ii)
yields $\hm^h(X\times Y)=0$, which is by
Theorem~\ref{basicHnull}(ii) enough.
\end{proof}

This theorem, together with the above example, provides an easy argument that
shows that consistently not every \smz{} set is \ssmz: The \smz{} set $X$
such that $X\times X$ is not \smz{} cannot be \ssmz{}.

We illustrate the power of the theorem by the following
\begin{coro}
Let $X\subs\Rset^2$. The following are equivalent.
\begin{enum}
\item $X$ is \ssmz,
\item all orthogonal projections of $X$ on lines are \ssmz,
\item at least two orthogonal projections of $X$ on lines are \ssmz.
\end{enum}
\end{coro}
\begin{proof}
Since orthogonal projections are uniformly continuous, (i)\Implies(ii) from
preservation of \ssmz{} by uniformly continuous mappings.
(ii)\Implies(iii) is trivial.
(iii)\Implies(i): Let $L_1,L_2$ be two nonparallel lines and $\pi_1,\pi_2$
the corresponding orthogonal projections. \emph{Mutatis mutandis} we may
suppose that $L_1$ is the $x$-axis and $L_2$ is the $y$-axis. Thus
$X\subs\pi_1X\times\pi_2X$.
Theorem \ref{productHUH}(ii) thus concludes the proof.
\end{proof}

%\subsection*{Corazza's model}
Theorem~\ref{productHUH}(ii) also raises the question whether a space whose
product with any \smz{} set of reals is \smz{} has to be \ssmz{}.
As shown in \cite{Zin_M-add}, the answer is consistently no:
The forcing extension constructed by Corazza in~\cite{MR982239}
we have the following.
A similar observation was noted without proof in~\cite{MR1610427}
and also in~\cite{MR3731016}.
\begin{prop}\label{corazza}
In the Corazza model there is a set $X\subs\Cset$ that is not \ssmz{}
and yet $X\times Y$ is \smz{} for each \smz{} set $Y\subs\Cset$.
\end{prop}

%%%%%%%%%%%%%%%%%%%%%%%%%%%%%%%%%%%%%%%%%%%%%%%%%%%%%%%%%%%%%%%%%
%%%%%%%%%%%%%%%%%%%%%%%%%%%%%%%%%%%%%%%%%%%%%%%%%%%%%%%%%%%%%%%%%
\section{Meager additive sets and sharp measure zero}
%%%%%%%%%%%%%%%%%%%%%%%%%%%%%%%%%%%%%%%%%%%%%%%%%%%%%%%%%%%%%%%%%
%%%%%%%%%%%%%%%%%%%%%%%%%%%%%%%%%%%%%%%%%%%%%%%%%%%%%%%%%%%%%%%%%

We now look at the meager-additive sets in Polish groups
and establish their surprising and profound connection with \ssmz{} sets.
The theory nicely parallels the Galvin-Mycielski-Solovay Theorem.
Most of the material of this section comes from \cite{Zin_M-additive} and
\cite{Zin_M-add}.

Whenever $\grG$ is a Polish group, $\ssmz(\grG)$ denotes the family of sharp
measure zero sets with respect to any left-invariant metric.
The notion od \ssmz{} is of course a uniform invariant -- it is neither a topological,
nor a metric property. Therefore it does not matter which left-invariant metric
we choose. The same proof shows that \ssmz{} sets, just like \smz{} sets,
form a bi-invariant \si ideal.

\begin{prop}\label{ssigma}
$\ssmz(\grG)$ is a bi-invariant \si-ideal.
\end{prop}

Recall that if $\grG$ is a Polish group, we denote by $\MM(\grG)$,
or simply by $\MM$ if there is no danger of confusion, the ideal of meager subsets
of $\grG$.

\begin{defn}
A set $S\subs\grG$ is called \emph{meager-additive} (or \emph{$\MM$-additive})
if $SM$ is meager for every meager set $M\subs\grG$.
The family of all meager-additive sets is denoted by $\MM^*(\grG)$.
\end{defn}

It is straightforward from the definition that
\begin{prop}\label{sssigma}
$\MM^*(\grG)$ is a bi-invariant \si-ideal.
\end{prop}

The hard implication of Galvin-Mycielski-Solovay Theorem claims that
if $S$ is \smz{} set in a locally compact group, then $S\cdot M\neq\grG$.
The analogous  statement for \ssmz{} and meager-additive sets is about
as hard as that.

\begin{thm}[{\cite{Zin_M-additive}}]\label{thm:action3}
Let $\grG$ be a locally compact Polish group.
Then
every \ssmz{} set $S\subs\grG$ is $\MM$-additive, i.e.,
$\ssmz(\grG)\subs\MM^*(\grG)$.
\end{thm}
\begin{proof}
The proof utilizes Lemma~\ref{comp}.
Suppose $S\subs\grG$ is a $\ssmz$ set and let $M\subs\grG$ be meager.
Let $K_n$ be compact sets in $\grG$ with $K_n\upto\grG$ and
let $P_n$ be compact nowhere dense sets with $P_n\upto M$.
Let $\{U_k\}$ be a countable base of $\grG$.
For each $k$ choose $x_0^k\in\grG$ and $\eps_0^k>0$ such that
$B(x_0^k,\eps_0^k)\subs U_k$ is compact, and let $C_k=B(x_0^k,\eps_0^k)$.
Use Lemma~\ref{comp} to recursively construct a sequence
$\seqeps$ of positive numbers such that
\begin{equation}\label{eq1}
\begin{split}
\forall n\ \forall i\leq n\
  & \forall x\in C_i\ \forall y\in K_n\ \exists z\in C_i\\
   & B(z,\eps_n)\subs B(x,\eps_{n-1})\setminus((B(y,\eps_n)\cap K_n)\cdot P_n).
\end{split}
\end{equation}
Since $S$ is $\ssmz$, there is an  $\gamma$-groupable cover $\{E_n\}$ of $S$
such that $\diam E_n<\eps_n$ for all $n$.
Hence for each $n$ there is $y$ such that $E_n\subs B(y,\eps_n)$.
Therefore we may use \eqref{eq1} to construct for each $k$ a sequence $\seq{x_n^k}$
such that
\begin{equation}\label{eq2}
  B(x_{n+1}^k,\eps_{n+1})\subs B(x_n^k,\eps_n)
  \setminus ((E_{n+1}\cap K_{n+1})\cdot P_{n+1}).
\end{equation}
It is easy to check that since $\{E_n\}$ is a $\gamma$-groupable cover of $S$
and $K_n\upto \grG$, the family $\{E_n\cap K_n\}$ is also a $\gamma$-groupable cover of $S$.
Thus we might have supposed that $E_n\subs K_n$, and also that all $E_n$'s are closed.
Therefore \eqref{eq2} simplifies to
\begin{equation}\label{eq3}
  B(x_{n+1}^k,\eps_{n+1})\subs B(x_n^k,\eps_n)
  \setminus (E_{n+1}\cdot P_{n+1}).
\end{equation}
In particular, $B(x_n^k,\eps_n)$ is a decreasing sequence of compact balls for all $k$
and thus there is a point $x^k\in U_k$ such that
\begin{equation}\label{eq4}
x^k\notin \bigcup_{n\in\Nset}(E_n\cdot P_n).
\end{equation}
Now construct a set $\widehat S$ as follows: Let $\mc G_j$ be the groups of $E_n$'s
witnessing to the $\gamma$-groupability of $\{E_n\}$.
Put $G_n=\bigcap_{n\in\mc G_j}E_n$ and let
$F_n=\bigcap_{i<n}G_i$ and $\widehat S=\bigcup_{n\in\Nset}F_n$.
It is clear that since $E_n$'s are closed, the set $\widehat S$ is $F_\sigma$, and
clearly $S\subs\widehat S$.
Moreover, routine calculation shows that
$\widehat S\times M\subs\bigcup_n E_n\times P_n$. Therefore \eqref{eq4} yields
$x^k\notin\widehat S\cdot M$ for all $k$.
So letting $D=\{x^k:k\in\Nset\}$, the set $D$ is disjoint with
$\widehat S\cdot M$ and it is dense in $X$.
Since $\widehat S$ and $M$ are \si compact, so is the set $\widehat S\times M$.
It follows that $\widehat S\cdot M$, being a continuous image of $\widehat S\times M$,
is also \si compact.
In summary, $\widehat S\cdot M)$ is an $F_\sigma$ set disjoint with a dense set,
and is thus meager.
\end{proof}

One would expect that the reverse implication that parallels the trivial
Proposition~\ref{prop:triv} of Prikry would be also very easy.
Surprisingly, it is not easy at all. Only very recently it was proved
in~\cite{Zin_M-additive} that it holds for Polish groups that admit a (both-sided)
invariant metric.

Recall that Polish groups that admit an invariant metric are referred to as \TSI{}
groups. Compact or abelian Polish groups are \TSI; any invariant metric on a
\TSI{} group is complete.

\begin{thm}[{\cite{Zin_M-additive}}]\label{polish1}
Let $\grG$ be a \TSI{} Polish group.
If $S\subs\grG$ is an $\MM$-additive set, then $S$ is $\ssmz$ (in any metric on $\grG$).
\end{thm}

\begin{thm}\label{M-additive}
Let $\grG$ be a locally compact \TSI{} Polish group.
Then $\ssmz(\grG)=\MM^*(\grG)$.
\end{thm}

It takes several pages to prove Theorem~\ref{polish1}, in contrast to the five lines
of the proof of Proposition~\ref{prop:triv}.
We present a proof of the particular case of $\grG=\Cset$ that takes advantage of
the regular combinatorial structure of the Cantor set
and a deep characterization by Shelah of $\MM$-additive sets
(\cite{MR1324470} or \cite[Theorem 2.7.17]{MR1350295})
and is thus much shorter.

\begin{lem}[{\cite{MR1324470}}]\label{ShelahM}
$X\subs\Cset$ is $\MM$-additive if and only if
\begin{multline*}
  \forall f\in\UPset\,\,\exists g\in\Pset\,\,\exists y\in\Cset\,\,
  \forall x\in X\,\,\forall^\infty n\,\,\exists k\\
  g(n)\leq f(k)<f(k+1)\leq g(n+1)\ \&\
  x\rest [f(k),f(k+1))=y\rest [f(k),f(k+1)).
\end{multline*}
\end{lem}
\begin{thm}\label{MeShelah}
$\ssmz(\Cset)=\MM^*(\Cset)$.
\end{thm}
\begin{proof}
Let $S\subs\Cset$ be $\MM$-additive. Let $h$ be a gauge.
We will show that $\uhm^h(S)=0$.
Define recursively $f\in\UPset$ subject to
$$
  2^{f(k)}\cdot h\bigl(2^{-f(k+1)}\bigr)\leq 2^{-k},\quad k\in\Nset.
$$
By Lemma~\ref{ShelahM} there is $g\in\Pset$ and $y\in\Cset$ such that
\begin{multline}\label{ShelahM2}
  \forall x\in X\,\,\forall^\infty n\,\,\exists k\\
  g(n)\leq f(k)<g(n+1)\ \&\
  x\rest [f(k),f(k+1))=y\rest [f(k),f(k+1)).
\end{multline}
%Recall that if $p\in\CCset$ then $\cyl p$ denotes the cone $\{f\in\Cset:p\subs f\}$.
Define
\begin{alignat*}{3}
  &\mc B_k&&=
  \bigl\{
  \cyl{p\concat y\rest [f(k),f(k+1))}:p\in 2^{f(k)}
  \bigr\},\qquad
  && k\in\Nset,\\
  &\mc G_n&&=\bigcup
  \bigl\{\mc B_k:g(n)\leq f(k)<g(n+1)\bigr\},
  && n\in\Nset,\\
 & \mc B&&=\bigcup_{k\in\Nset}\mc B_k=\bigcup_{n\in\Nset}\mc G_n.
\end{alignat*}
With this notation~\eqref{ShelahM2} reads
\begin{equation}\label{ShelahM22}
  \forall x\in X\,\,\forall^\infty n\,\,\exists G\in\mc G_n\,\,
  x\in G.
\end{equation}
Since each of the families $\mc G_n$ is finite,
it follows that $\mc G_n$'s witness that $\mc B$ is a $\gamma$-groupable
cover of $X$.
Using Lemma~\ref{gammagr} (and Lemma~\ref{lem1}(iv)) it remains to show
that the Hausdorff sum $\sum_{B\in\mc B}h(\diam B)$ is finite.
Note that the cones forming the families $B\in\mc B_k$
are actually balls of radius $2^{-f(k+1)}$,
i.e., $\diam B=2^{-f(k+1)}$ for all $k$ and all $B\in\mc B_k$.
Each of the cones is determined by one $p\in 2^{f(k)}$, therefore
$\abs{\mc B_k}=2^{f(k)}$.
Overall we have
\begin{equation*}
  \sum_{B\in\mc B}h(\diam B)=
  \sum_{k\in\Nset}\sum_{B\in\mc B_k}h(\diam B)=
  \sum_{k\in\Nset}2^{f(k)}\cdot h(2^{-f(k+1)})
  \leq\sum_{k\in\Nset}2^{-k}<\infty.
  \qedhere
\end{equation*}
\end{proof}

The most interesting question raised by Theorems~\ref{thm:action3} and~\ref{polish1}
is of course if the analogue of Conjecture~\ref{CHSMZ} holds for \ssmz.
\begin{question}[CH]
Is it true that if the inclusion $\ssmz(\grG)\subs\MM^*(\grG)$ holds for
a Polish group $\grG$, then $\grG$ is locally compact?
\end{question}
Another open problem is the r\^{o}le of \TSI{} in Theorem~\ref{M-additive}.
\begin{question}
Can the \TSI{} assumption in Theorem~\ref{M-additive} be dropped, or weakened to \CLI?
\end{question}

\subsection*{Continuous images and cartesian products of meager-additive sets}

Since meager-additive sets coincide with \ssmz{} sets in \TSI{} locally compact groups,
they are preserved by continuous mappings
and by cartesian products.
\begin{thm}\label{mapping22}
Let $\grG_1$ be a \TSI{} Polish group and
$\grG_2$ a locally compact Polish group.
Let $f:\grG_1\to\grG_2$ a continuous mapping.
If $X\subs\grG_1$ is $\MM$-additive, then so is $f(X)$.
\end{thm}
%\begin{proof}
%Since $X$ is $\MM$-additive it is $\ssmz$ by Theorem~\ref{polish1}.
%By Proposition~\ref{mapping1}, $f(X)$ is $\ssmz$ as well.
%By Theorem~\ref{thm:action3}, $f(X)$ is $\MM$-additive.
%\end{proof}

\begin{thm}\label{cart}
Let $\grG_1,\grG_2$ be \TSI{} locally compact Polish groups.
Let $X_1\subs\grG_1$, $X_2\subs\grG_2$.
\begin{enum}
\item If $X_1$ is \smz{} and $X_2$ is $\MM$-additive, then $X_1\times X_2$ is \smz.
\item If $X_1$ and $X_2$ are $\MM$-additive, then so is $X_1\times X_2$.
\end{enum}
\end{thm}
%\begin{proof}
%(i)
%By Theorem \ref{main3a}, $X_2$ is $\ssmz$. By
%\cite[Theorem 3.14]{Zin_M-add}, a product of a $\smz$ and $\ssmz$ sets is $\smz$.
%
%(ii)
%By Theorem \ref{main3a}, both $X_1$ and $X_2$ are $\ssmz$. By
%\cite[Theorem 3.14]{Zin_M-add}, a product of two $\ssmz$ sets is $\ssmz$.
%Now apply Theorem \ref{main3a} to conclude that $X_1\times X_2$ is
%$\MM$-additive.
%\end{proof}

\begin{coro}
Let $\grG_1,\grG_2$ be \TSI{} locally compact Polish groups.
Let $X\subs\grG_1\times\grG_2$.
The following are equivalent.
\begin{enum}
\item $X$ is $\MM$-additive,
\item $\proj_1X$ and $\proj_2X$ are $\MM$-additive,
\item $\proj_1X\times\proj_2X$ is $\MM$-additive.
\end{enum}
\end{coro}
%\begin{proof}
%Write $X_1=\proj_1 X$, $X_2=\proj_2 X$.
%(i)\Implies(ii) Suppose $X$ is $\MM$ additive.
%Then by Theorem~\ref{mapping22} both $X_1$ and $X_2$ are $\MM$ additive.
%(ii)\Implies(iii) Suppose $X_1$ and $X_2$ are $\MM$ additive.
%Then by Theorem~\ref{cart} $X_1\times X_2$ is $\MM$ additive.
%\end{proof}

We conclude this section with a few remarks on meager-additive sets in the Cantor set
$\Cset$. There are a few variations of $\MM$-additivity:
A set $S$ in $\Cset$ is
\begin{itemyze}
\item \emph{$\MM$-additive} if $\forall M\in\MM\quad S+M\in\MM$,
\item \emph{sharply $\MM$-additive} if $\forall M\in\MM\
  \exists F\sups S\ F_\sigma\quad F+M\in\MM$,
\item \emph{flatly $\MM$-additive} if $\forall M\in\MM\
  \exists F\sups S\ F_\sigma\quad F+M\neq\Cset$,
\item \emph{$\EE$-additive} if $\forall E\in\EE\ S+E\in\EE$,
\item \emph{sharply $\EE$-additive} if $\forall E\in\EE\
  \exists F\sups S\ F_\sigma\quad F+E\in\EE$.
\end{itemyze}

The question whether $\EE$-additive sets are related to $\MM$-additive sets
are related was posed by Nowik and Weiss~\cite{MR1905154}. Their question
was answered in~\cite{Zin_M-add} by the following theorem.
Let us note that while one would expect that, e.g., the proof of (ii)\Implies(iv)
is a matter of routine, it is actually surprisingly difficult.

\begin{thm}[{\cite{Zin_M-add}}]
Let $S\subs\Cset$. The following properties of $S$ are equivalent.
\begin{enum}
\item $S$ is \ssmz{},
\item $S$ is $\MM$-additive,
\item $S$ is sharply $\MM$-additive,
\item $S$ is flatly $\MM$-additive,
\item $S$ is $\EE$-additive,
\item $S$ is sharply $\EE$-additive.
\end{enum}
\end{thm}

The definitions of these notions extend in a straightforward way to
Polish groups (or in case $\EE$ is considered, locally copact Polish groups).
However, the proofs of the above theorem
depends very much on the fine combinatorial structure of $\Cset$
and are thus not easily transferable to a context of a Polish group.
With quite some effort the equivalence
(i)\equival(ii)\equival(iii)\equival(iv) in \TSI{} locally compact Polish groups
was proved in~\cite{Zin_M-additive}.
(The equivalence (i)\equival(ii) is presented in this section as
Theorem~\ref{M-additive}.) But the equivalences including $\EE$ are still not
understood.

\begin{question}[{\cite{Zin_M-additive}}]
Let $\grG$ be a locally compact \TSI{} Polish group.
\begin{enum}
\item Is every $\MM$-additive set in $\grG$ $\EE$-additive?
\item Is every $\EE$-additive set in $\grG$ $\MM$-additive?
\item Is every $\EE$-additive set in $\grG$ sharply $\EE$-additive?
\end{enum}
\end{question}

%%%%%%%%%%%%%%%%%%%%%%%%%%%%%%%%%%%%%%%%%%%%%%%%%%%%%%%%%%%%%%%%%
%%%%%%%%%%%%%%%%%%%%%%%%%%%%%%%%%%%%%%%%%%%%%%%%%%%%%%%%%%%%%%%%%
\section{Uniformity number of meager-additive sets}
\label{sec:uniformity}
%%%%%%%%%%%%%%%%%%%%%%%%%%%%%%%%%%%%%%%%%%%%%%%%%%%%%%%%%%%%%%%%%
%%%%%%%%%%%%%%%%%%%%%%%%%%%%%%%%%%%%%%%%%%%%%%%%%%%%%%%%%%%%%%%%%

Since the notion of sharp measure zero is rather new, not much is known about
the cardinal invariants of the \si ideal of \ssmz{} sets.
We will investigate only the uniformity number of \ssmz{} for metric
spaces and Polish groups. We refer to section \ref{sec:cardinal} for the
notation.

Bartoszy\'nski \cite{MR1350295} and Pawlikowski \cite{MR776210}
investigated and calculated a related cardinal --
the uniformity number of the ideal of meager-additive sets in $\Cset$.
This cardinal invariant is termed \emph{transitive additivity of $\MM$}
and denoted by $\addT$. By \cite[2.7.14]{MR1350295}, $\addT=\eqq$.

The two results on Polish groups, Theorem~\ref{invariants} and
Corollary~\ref{invariants2}, come from \cite{Zin_M-add}.
The other results of this section are new.

\begin{thm}\label{thm:Roth2}
Let $X$ be a separable metric space that is not \ssmz{}.
\begin{enum}
\item $\nonSS \Pset=\addM$,
\item $\nonSS \Cset=\eqq$,
\item $\addM\leq\nonSS X$,
\item if $X$ is \si totally bounded, then $\eqq\leq\nonSS X$,
\item if $X$ is not of universal measure zero, then $\nonSS X\leq\nonN$.
\end{enum}
\end{thm}
\begin{proof}
(i)
Recall that the metric on $\Pset$ is the least difference metric
given by $d(f,g)=2^{-n(f,g)}$, where $n(f,g)=\min\{n:f(n)\neq g(n)\}$.

Let $S\subs\Pset$ be an unbounded set such that $\abs X=\bbb$.
Since $X$ is not bounded, it is not \si totally bounded
and in particular it is not \ssmz.
It follows that $\non\ssmz(\Pset)\leq\bbb$.

Theorem \ref{thm:Roth}(iii) yields a set $S\subs\Pset$ that is not $\smz$
and $\abs S=\cov\MM$. This set is clearly not $\ssmz$ and thus
$\non\ssmz(\Pset)\leq\cov\MM$.

It follows that $\non\ssmz(\Pset)\leq\min(\cov\MM,\bbb)$. By a theorem of
Miller~\cite{MR613787}
$\min(\cov\MM,\bbb)=\addM$ (see also \cite[2.2.9]{MR1350295}).
Thus we have $\non\ssmz(\Pset)\leq\addM$.

In the other direction, suppose that $S\subs\Pset$ is not \ssmz{}.
By Theorem~\ref{polish1}, $S$
is not meager-additive set in the group $\ZZset$.
Therefore there is a meager set $M\subs\ZZset$ such that
$S+M=\bigcup_{s\in S}(M+s)$ is not meager and hence
$\abs S\geq\addM$. Thus $\non\ssmz(\Pset)\geq\addM$.

(ii)
By Theorem~\ref{thm:action3}, a set $S\subs\Cset$ is \ssmz{} if and only if
it is meager-additive. Thus $\non\ssmz(\Cset)=\addT$ and the latter equals
by the aforementioned result of Bartoszy\'nski to $\eqq$.

(iii)
Let $\{z_m:m\in\Nset\}\subs X$ be a dense set in $X$.
To each $x\in X$ assign a function $\widehat x\in\Pset$ defined by
\begin{equation}\label{sbpmap}
  \widehat x(n)=\min\{m:d(x,z_m)<2^{-n}\}.
\end{equation}
We claim that the inverse map $\widehat x\mapsto x$ is Lipschitz. Indeed,
if $d(\widehat x,\widehat y)=2^{-n}$, then $\widehat x(n-1)=\widehat y(n-1)$
and in particular $\exists m\ d(x,z_m)<2^{-n+1}$ and $d(y,z_m)<2^{-n+1}$.
Therefore $d(x,y)<2\cdot 2^{-n+1}$, whence $d(x,y)<4d(\widehat x,\widehat y)$.
In particular $\widehat x\mapsto x$ is uniformly continuous.
So if $S\subs X$ is not \ssmz, then $\widehat S=\{\widehat x:x\in S\}$
is not \ssmz{} as well. If follows that $\nonSS X\geq\nonSS\Pset$
and (iii) follows from (i).

%Now consider the Baer-Specker group $\ZZset$.
%By Theorem \ref{thm:action3}, a set $S\subs\ZZset$ is \ssmz{} if and only if
%it is meager-additive. Therefore $\nonSS{\ZZset}=\non(\MM^*(\ZZset))
%\geq\non(\MM(\ZZset))=\addM$.
%Since clearly $\nonSS{\ZZset}=\nonSS\Pset$, we are done.

(iv)
Consider the completion $X^*$ of $X$. Since $X$ is \si totally bounded, there
is a \si compact set $K\subs X^*$ that contains $X$. Let
$K_n\upto K$ be a sequence of compact sets. Suppose that $S\subs X$ is
a not $\ssmz$ set such that $\abs S=\nonSS X$.
There is $n$ such that $S\cap K_n$ is not $\ssmz$, therefore
$\abs{S\cap K_n}=\nonSS X$. Since $K_n$ is compact, it is \emph{dyadic}:
there is a uniformly continuous mapping $f:\Cset\to K_n$ onto $K_n$.
For each $x\in S\cap K_n$ pick $\widehat x\in f^{-1}(x)$ and set
$\widehat S=\{\widehat x:x\in S\cap K_n\}$.
Then $f(\widehat S)=S\cap K_n$, hence $\widehat S$ is
not $\ssmz$, and clearly $\abs{\widehat S}=\abs{S\cap K_n}=\nonSS X$.
It follows that $\nonSS\Cset\leq\nonSS X$ whence $\addT\leq\nonSS X$ by (ii).

(v) is a trivial consequence of Theorem~\ref{thm:Roth}(ii).
\end{proof}

As expected, for analytic metric spaces we can do better.
Recall that a metric space has the \emph{small ball property}
if it admits a base $\{B_n\}$ such that $\diam B_n\to 0$.
This notion is due to Behrends and Kadec \cite{MR1880727}.
We refer to \cite{MR3453581} for more information.

\begin{thm}\label{thm:Roth3}
Let $X$ be an uncountable analytic metric space.
\begin{enum}
\item $\addM\leq\nonSS X\leq\eqq$,
\item if $X$ is \si totally bounded, then $\nonSS X=\eqq$,
\item if $X$ does not have the small ball property, then $\nonSS X=\addM$.
\end{enum}
\end{thm}
\begin{proof}
(i) The left-hand inequality is Theorem~\ref{thm:Roth2}(iii).
Since $X$ contains a (uniform) copy of $\Cset$, Theorem~\ref{thm:Roth2}(ii)
yields $\eqq=\nonSS\Cset\geq\non\ssmz(\grG)$.

(ii) follows from (i) and Theorem~\ref{thm:Roth2}(iv).

(iii) Consider the mapping $x\mapsto\widehat x$ defined by \eqref{sbpmap}.
Let $B\subs\Pset$ be an unbounded set such that $\abs B=\bbb$.
As shown in~\cite[4.4]{MR3453581}, the set $\widehat X=\{\widehat x:x\in X\}$
is dominating in $\Pset$. Therefore for
each $f\in B$ there is $x_f\in X$ such that $\widehat x_f\geq^*f$.
Set $S=\{x_g:f\in B\}$. Then $\widehat S$ is not bounded, because
it dominates $B$. It follows that $S$ is not \si totally bounded in $X$,
and in particular it is not \ssmz.
Since clearly $\abs{\widehat S}=\abs B=\bbb$, we conclude that
$\nonSS X\leq\bbb$. By (i) also $\nonSS X\leq\addM$, so
$\nonSS X\leq\min(\eqq,\bbb)=\addM$, and the reverse inequality follows from~(i).
\end{proof}

We will now calculate uniformity numbers of $\ssmz(\grG)$ for \CLI{} Polish groups
and of $\MM^*(\grG)$ for \TSI{} Polish groups.

\begin{thm}[{\cite{Zin_M-add}}]\label{invariants}
Let $\grG$ be a \CLI{} Polish group.
\begin{enum}
\item If $\grG$ is locally compact, then $\nonSS\grG=\eqq$.
\item If $\grG$ is not locally compact, then $\nonSS\grG=\addM$.
\end{enum}
\end{thm}
\begin{proof}
(i) is a straightforward from Theorem~\ref{thm:Roth3}(ii).

(ii) By Lemma~\ref{lem:Best}, $\grG$ contains a uniform copy of $\Pset$.
Therefore $\nonSS\grG\leq\nonSS\Pset$. Now apply Theorem~\ref{thm:Roth2}(i)
to get $\nonSS\grG\leq\nonM$. The reverse inequality is straightforward from
Theorem~\ref{thm:Roth3}(i).
\end{proof}

The following is a trivial consequence of this theorem and Theorem~\ref{thm:action3}.
\begin{coro}[{\cite{Zin_M-add}}]\label{invariants2}
Let $\grG$ be a \TSI{} Polish group.
\begin{enum}
\item If $\grG$ is locally compact, then $\MM^*(\grG)=\eqq$.
\item If $\grG$ is not locally compact, then $\MM^*(\grG)=\addM$.
\end{enum}
\end{coro}

We conclude with a simple argument that shows that $\ssmz$
is consistently not a topological property.
The Baer-Specker group $\ZZset$ is a \TSI{} Polish group.
On the other hand, it is homeomorphic to the set
of irrational numbers, so regard it as a subset of $\Rset$.
By Theorem~\ref{invariants}(ii) there is a set $X\subs\ZZset$ such that
$\abs X=\addM$
and that is not $\ssmz$ in the invariant metric.
On the other hand, if $\addM<\add^*\MM$, then $X$ is by Theorem~\ref{invariants}(i)
$\ssmz$ in the metric of the real line. Since $\addM<\add^*\MM$
is relatively consistent with ZFC (as proved by Pawlikowski~\cite{MR776210}),
$X$ is consistently a set that is $\ssmz$ is one metric on $\ZZset$
and not $\ssmz$ in another homeomorphic metric.

%%%%%%%%%%%%%%%%%%%%%%%%%%%%%%%%%%%%%%%%%%%%%%%%%%%%%%%%%%%%%%%%%
%%%%%%%%%%%%%%%%%%%%%%%%%%%%%%%%%%%%%%%%%%%%%%%%%%%%%%%%%%%%%%%%%
% BIBLIOGRAPHY ----------------------------------------------------
\bibliographystyle{amsplain}
\bibliography{smz,M-add,M-additive}

\end{document}